\documentclass[DIV=calc, paper=letter, fontsize=11pt]{scrartcl}	 

\usepackage[]{hyperref}

\usepackage[]{units}
\usepackage{lipsum} 
\usepackage[english]{babel} 
\usepackage[protrusion=true,expansion=true]{microtype} 
\usepackage{amsmath,amssymb,amsfonts,amsthm} 

\usepackage[svgnames, table]{xcolor} 
\usepackage[hang, small,labelfont=bf,up,textfont=it,up]{caption} 
\usepackage{booktabs} 
\usepackage{fix-cm}	 
\usepackage[]{units}
\usepackage{rotating}
\usepackage{tabularx} 
\usepackage{setspace}
\usepackage{fullpage}
\usepackage{sectsty} 

\usepackage{fancyhdr} 
\usepackage{lastpage} 

\usepackage{nicefrac}
\usepackage{cite}

\usepackage{tikz}
\usetikzlibrary{calc,patterns,decorations.pathmorphing,decorations.markings}
\usepackage{pgfplots}

\usepackage{multicol}
\usepackage{float} %

\newcounter{tempcolnum}
\makeatletter
\newcommand{\multicolinterrupt}[1]{
\setcounter{tempcolnum}{\col@number}
\end{multicols}
#1%
\begin{multicols}{\value{tempcolnum}}
}
\makeatother

\theoremstyle{definition}
\newtheorem{definition}{Definition}[section]

\theoremstyle{theorem}
\newtheorem{theorem}{Theorem}[section]

\theoremstyle{proposition}
\newtheorem{proposition}{Proposition}[section]

\theoremstyle{lemma}
\newtheorem{lemma}{Lemma}[section]

\theoremstyle{corollary}
\newtheorem{corollary}{Corollary}[section]

\theoremstyle{remark}
\newtheorem{remark}{Remark}[section]

\usepackage{graphicx}
\DeclareGraphicsExtensions{.png}

\usepackage{lettrine} 

\usepackage{caption}
\usepackage{graphicx}

\usepackage{titling} 

\newcommand{\HorRule}{\color{DarkGoldenrod} \rule{\linewidth}{1pt}} 

\pretitle{\vspace{-80pt} \begin{flushleft} \HorRule \fontsize{18}{18} \color{DarkRed} \selectfont} 

\title{\LARGE Knot polynomials of open and closed curves}

\posttitle{\par\end{flushleft}} 
\preauthor{\vspace{-5pt} \begin{flushleft}\fontsize{14}{14} \large  \color{DarkRed}} 
\author{Eleni Panagiotou$^{\#,*}$ and Louis H. Kauffman$^{\S}$ \\} 
\postauthor{\fontsize{12}{12} \small \color{Black} 
$^{\#}$Department of Mathematics and SimCenter, University of Tennessee at Chattanooga, TN 37403, USA, $<$eleni-panagiotou@utc.edu$>$  \\
$^{\S}$Department of Mathematics, Statistics and Computer Science
University of Illinois at Chicago
Chicago, IL 60607-7045, USA
and
Department of Mechanics and Mathematics
Novosibirsk State University
Novosibirsk
Russia
$<$kauffman@uic.edu$>$ \\
\lfoot{\footnotesize * Work supported by NSF DMS - 1913180}
\end{flushleft} \vspace{-18pt} \HorRule} 

\definecolor{issuePJA_color}{rgb}{1.0,0.0,0.0}

\definecolor{commentPJA_color}{rgb}{1.0,0.0,0.8}

\definecolor{commentEP_color}{rgb}{1.0,0.0,0.8}

\DeclareGraphicsExtensions{.png}

\begin{document}

\maketitle 

\thispagestyle{fancy} 


\textbf{In this manuscript we introduce a method to measure entanglement of curves in 3-space that extends the notion of knot and link polynomials to open curves. We define the bracket polynomial of curves in 3-space and show that it has real coefficients and is a continuous function of the chain coordinates.  This is used to define the Jones polynomial in a way that it is applicable to both open and closed curves in 3-space. For open curves, it has real coefficients and it is a continuous function of the chain coordinates and as the endpoints of the curve tend to coincide, the Jones polynomial of the open curve tends to that of the resulting knot. For closed curves, it is a topological invariant, as the classical Jones polynomial. We show how these measures attain a simpler expression for polygonal chains and provide a finite form for their computation in the case of chains of 3 and 4 edges.}


\section{Introduction}


Open curves in space can entangle and even tie knots, a situation that arises in many physical systems of filaments, such as polymers, textiles, chemical compounds \cite{Arsuaga2005,Virnau2006,Taylor1974,Sumners1990,Sulkowska2012,Qin2011,Kroeger2005,Tzoumanekas2006,Everaers2004,Delgado2017,Liu2018}. In different contexts, entanglement of filaments affects material properties, function or other aspects related to fluid mechanics, biology, chemistry or engineering \cite{Edwards1967,deGennes1979,Foster2019,Sulkowska2012,Liu2018}. To measure entanglement of open curves it is natural to look for measures of complexity in the study of knots and links \cite{Kauffman2001}. Even though many strong and refined measures of topological complexity for knots and links have been created in the last century, such as knot and link polynomials \cite{Jones1985,Jones1987,Kauffman1987,Kauffman1990,Przytycki1987,Freyd1985}, the only one that is sensitive on the configurations of open curves is the Gauss linking integral (introduced in 1877) \cite{Gauss1877}. In this work we define knot and link polynomials of \textit{open chains in 3-space}. To do this, we combine ideas of the Gauss linking integral and the notion of knotoids (open chain \textit{diagrams} \cite{Turaev2012,Gugumcu2019,Gugumcu2017,Gugumcu2017b}).

 A knot is a simple closed curve in space. Similarly, a link is formed by many simple closed curves in space that do not intersect each other.  Two knots or links are equivalent if one can be continuously deformed to the other without allowing cutting and pasting. A topological invariant is a function over the space of knots or links that is invariant under such deformations \cite{Freyd1985,Przytycki1987,Kauffman1990}. When dealing with open chains, the above notion of topological equivalence is not useful, since any mathematical open curve can be deformed to another without cutting and pasting. In fact, one does not need a measure of complexity of open chains that is invariant under deformations, but rather a measure that varies continuously in the space of configurations. Such a measure is the Gauss linking integral. For two closed chains, the Gauss linking integral is an integer topological invariant that measures the algebraic number of times one chain turns around the other. For two open chains, it is a real number that is a continuous function of the chain coordinates. The Gauss linking integral has been very useful in measuring entanglement in physical systems of open or closed filaments \cite{Baiesi2017,Rogen2003,Panagiotou2019,Panagiotou2013b,Panagiotou2014,Panagiotou2015}. However, more refined measures of entanglement of one, two or more components, are needed. In this direction several approximation efforts have appeared, aiming at mapping an open chain to a knot type, or a knotoid type \cite{Sulkowska2012,Goundaroulis2017,Goundaroulis2017b}.

In this manuscript we introduce a new measure of entanglement of open chains in 3-space that is a well-defined function of the chain coordinates in 3-space that does not approximate an open chain by any particular closed chain or any particular projection of the open chain. Namely, we define the \textit{bracket polynomial of open curves in 3-space}, a polynomial with real coefficients which is a continuous function of the chain coordinates. This is used to define the \textit{Jones polynomial of open chains in 3-space}. The Jones polynomial of open 3-dimensional chains is a continuous function of the chains coordinates and, as the endpoints of the chains tend to coincide, it tends to the Jones polynomial of the resulting knot, a topological invariant of the knot. We stress that this is the first well defined new measure of entanglement of open chains that is a continuous measure of complexity of open curves since the Gauss linking integral and it is stronger than the Gauss linking integral. 

An important reason why the Gauss linking integral has been very useful in applications is that a finite form for its computation exists that avoids numerical integration \cite{Banchoff1976}. To this direction, in this manuscript we also provide a finite form for the  computation of the bracket and Jones polynomials in the case of a polygonal chain of 3 and 4 edges (open or closed). This is the base case upon which the general case of more edges will be studied in a sequel to this paper.

The manuscript is organized as follows: Section \ref{measures} discusses background information on measures of entanglement, Section \ref{bracketopen} gives the definition and properties of the bracket polynomial of open chains in 3-space and uses the bracket polynomial of open chains to define the Jones polynomial of open chains. We stress that, even though in this manuscript we focus on single open chains, all the definitions and properties of those described in Section \ref{bracketopen}  apply to a collection of open chains. Sections \ref{finite} and \ref{Jonesfinite} provide a finite formula for the computation of the bracket and  Jones  polynomials of polygonal chains of 3 and 4 edges.

\section{Measures of complexity of open chains and their projections}\label{measures}

In this Section, we provide background information that is necessary for the rest of the manuscript. More precisely, we discuss the Gauss linking integral, a measure of entanglement of both open and closed \textit{3-dimensional curves} and the bracket and Jones polynomial of knotoids, a measure of complexity of open knot \textit{diagrams} (projections of open 3-dimensional chains).

\subsection{The Gauss linking integral}

A measure of the degree to which polymer chains interwind and attain complex configurations is the Gauss linking integral:

\begin{definition}\label{lk} (Gauss Linking Number). The Gauss \textit{Linking Number} of two disjoint (closed or open) oriented curves  $l_1$ and $l_2$, whose arc-length parametrizations are $\gamma_1(t),\gamma_2(s)$ respectively, is defined as the following double integral over $l_1$ and $l_2$ \cite{Gauss1877}:

\begin{equation}\label{Gausslk}
L(l_1,l_2)=\frac{1}{4\pi}\int_{[0,1]}\int_{[0,1]}\frac{(\dot\gamma_1(t),\dot\gamma_2(s),\gamma_1(t)-\gamma_2(s))}{||\gamma_1(t)-\gamma_2(s)||^3}dt ds,
\end{equation}

\noindent where $(\dot\gamma_1(t),\dot\gamma_2(s),\gamma_1(t)-\gamma_2(s))$
is the \textit{scalar triple product} of $\dot\gamma_1(t),\dot\gamma_2(s)$ and
$\gamma_1(t)-\gamma_2(s)$.
\end{definition}

For closed chains, the Gauss linking integral is equal to the half algebraic sum of crossings of the two chains in any projection direction, it is an integer and a topological invariant of the link. 

For open chains, the Gauss linking integral is equal to the average of half the algebraic sum of crossings between the projections of the two chains over all possible projection directions. It is a real number and a continuous function of the chain coordinates. 

The Gauss linking integral can be applied over one curve, to measure its self-entanglement, called writhe, we denote $Wr$. By taking the absolute value of the integrand the writhe becomes the average crossing number, we denote $ACN$.

\subsubsection{Finite form of the Gauss linking integral}\label{Gaussfinite}

In \cite{Banchoff1976}, a finite form for the Gauss linking integral of two edges was introduced, which gives a finite form for the Gauss linking integral over one or two polygonal chains.

Let $E_n,R_m$ denote two polygonal chains of edges $e_i,i=1,\dotsc,n, r_j,j=1,\dotsc,m$, then 

\begin{equation}\label{Gausslk}
L(E_n,R_m)=\sum_{i=1}^n\sum_{j=1}^mL(e_i,r_j)
\end{equation}

\noindent where $L(e_i,r_j)$ is the Gauss linking integral of two edges. Let $e_i$ be the edge that connects the vertices $\vec{p}_i,\vec{p}_{i+1}$ and $r_j$ be the edge that connects the vertices $\vec{p}_j,\vec{p}_{j+1}$ (see Figure \ref{fig:banchoff} for an illustrative example). In \cite{Banchoff1976} it was shown that $L(e_i,r_j)=\frac{1}{4\pi}Area(Q_{i,j})$, where $Q_{ij}$ for $i<j$ denotes the quadrangle defined by the faces of the quadrilateral formed by the vertices $\vec{p}_i,\vec{p}_{i+1},\vec{p}_{j},\vec{p}_{j+1}$. This area can be computed by adding the dihedral angles of this quadrilateral. The faces of this quadrangle have normal vectors $\vec{n}_i,i=1,\dotsc 4$, defined as follows \cite{Klenin2000}:

\begin{figure}[h]
   \begin{center}
     \includegraphics[width=0.75\textwidth]{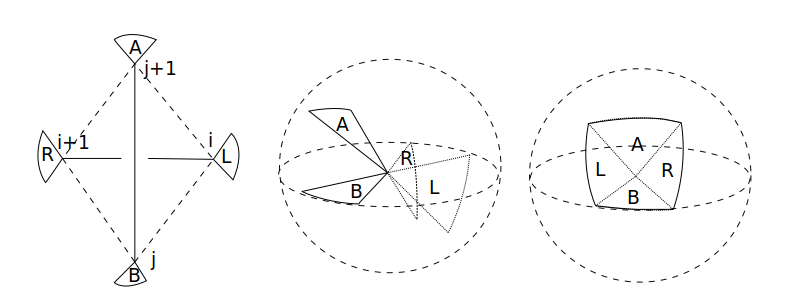}
      \end{center}
     \caption{The area of the quadrangle is bounded by the great circles with normal vectors $\vec{n}_1,\vec{n}_2,\vec{n}_3,\vec{n}_4$, determined by the faces of the quadrilateral. In fact, the quadrangle is formed by gluing together, with correct orientation the tiles $A,B,R,L$. The vectors vectors $\vec{n}_1,\vec{n}_2,\vec{n}_3,\vec{n}_4$ are perpendicular to the tiles $L,A,R$ and $B$ respectively, pointing outwards of the tetrahedron for $A,B$ and inwards for $L,R$. These tiles define a quadrangle with faces $A,L,B,R$ in the counterclockwise orientation, with all the normal vectors pointing outside the quadrangle. }
     \label{fig:banchoff}
\end{figure}

\begin{align}
&    \vec{n}_1=\frac{\vec{r}_{i,j}\times \vec{r}_{i,j+1}}{||\vec{r}_{ij}\times \vec{r}_{i,j+1}||}\nonumber,    \vec{n}_2=\frac{\vec{r}_{i,j+1}\times \vec{r}_{i+1,j+1}}{||\vec{r}_{i,j+1}\times \vec{r}_{i+1,j+1} ||}\nonumber   \vec{n}_3=\frac{\vec{r}_{i+1,j+1}\times \vec{r}_{i+1,j}}{|| \vec{r}_{i+1,j+1}\times \vec{r}_{i+1,j} ||}\nonumber,    \vec{n}_4=\frac{\vec{r}_{i+1,j}\times \vec{r}_{i,j}}{|| \vec{r}_{i+1,j}\times \vec{r}_{i,j}||}\nonumber
\end{align}

\noindent where $\vec{r}_{ij}=\vec{p}_i-\vec{p}_j$, $\vec{r}_{i,j+1}=\vec{p}_i-\vec{p}_{j+1}$, $\vec{r}_{i+1,j}=\vec{p}_{i+1}-\vec{p}_j$, $\vec{r}_{i+1,j+1}=\vec{p}_{i+1}-\vec{p}_{j+1}$.

The area of the quadrangle $Q_{ij}$ is: $Area(Q_{ij})=\arcsin(\vec{n}_1\cdot \vec{n}_2)+\arcsin(\vec{n}_2\cdot \vec{n}_3)+\arcsin(\vec{n}_3\cdot \vec{n}_4)+\arcsin(\vec{n}_4\cdot \vec{n}_1)$.

\subsection{The bracket polynomial of knotoids}\label{bracket}

The theory of knotoids was introduced by V.Turaev \cite{Turaev2012} in 2012 (see also \cite{Gugumcu2017}).  Knotoids are  open  ended  knot  diagrams (see Figure \ref{fig:knotoids}). Three  \textit{Reidemeister  moves} (see Figure \ref{fig:omega}),  are  defined  on knotoid  diagrams by  modifying  the  diagram within small surrounding disks that do  not  utilize  the  endpoints (forbidden moves shown in Figure 4).  
Two knotoid diagrams are said to be equivalent if they are related to each other by a finite sequence of such moves (and isotopy of $S^2$, $R^2$ for knotoid diagrams in $S^2$, $R^2$, respectively). 

\begin{figure}[h]
   \begin{center}
     \includegraphics[width=0.45\textwidth]{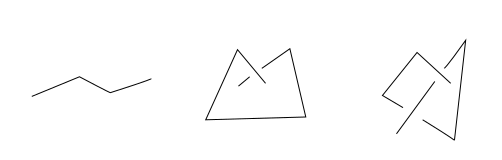}
      \end{center}
     \caption{ Examples of (polygonal) knotoids (open simple arc diagrams). Notice that knotoids refer to \textit{projections of open chains}, while knots refer to closed chains in 3-space.}
     \label{fig:knotoids}
\end{figure}

\begin{figure}[h]
   \begin{center}
     \includegraphics[width=0.8\textwidth]{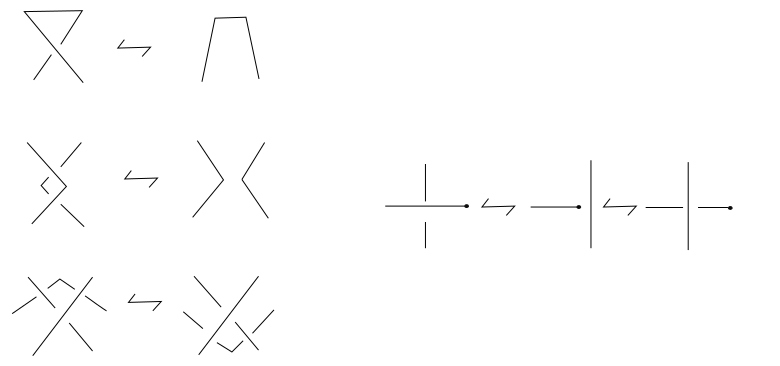}
      \end{center}
     \caption{Left: The Reidemeister moves for knotoids and Right: forbidden knotoid moves.}
     \label{fig:omega}
  
\end{figure}

The bracket polynomial of knotoids in $S^2$ or $R^2$  is defined by extending the state expansion of the bracket polynomial of knots. The following initial conditions and diagrammatic equations are sufficient for the skein computation of the bracket polynomial of classical knotoids:

\begin{equation}\label{skein}
\langle\raisebox{-10pt}{\includegraphics[width=.05\linewidth]{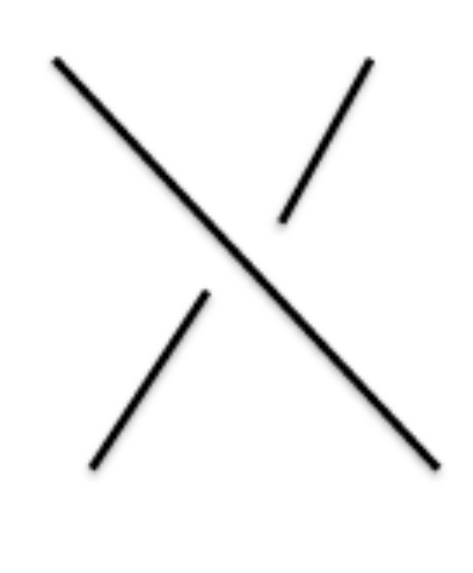}}\rangle=A\langle\raisebox{-10pt}{\includegraphics[width=.05\linewidth]{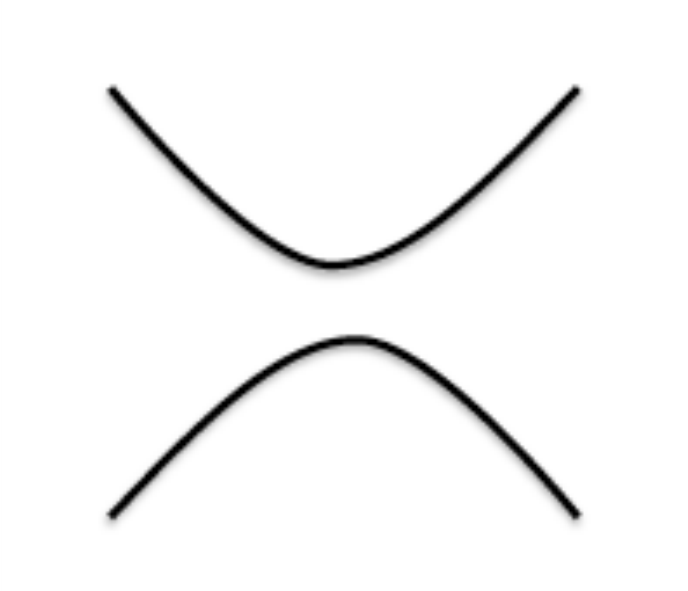}}\rangle+A^{-1}\langle\raisebox{-10pt}{\includegraphics[width=.05\linewidth]{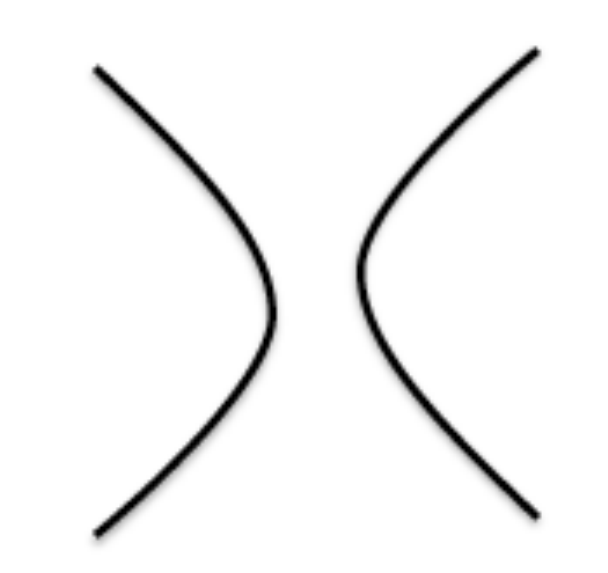}}\rangle,\hspace{0.5cm}\langle K\cup \bigcirc\rangle=(-A^2-A^{-2})\langle K\rangle,\hspace{0.5cm} \langle\includegraphics[width=.05\linewidth]{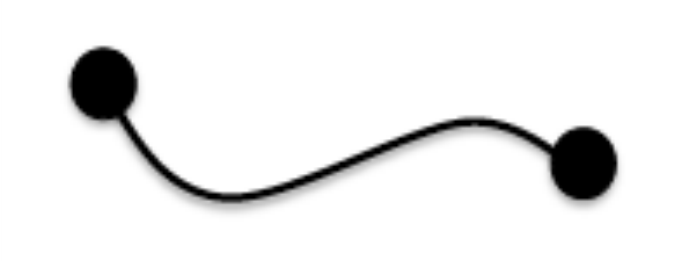}\rangle=1.
\end{equation}

\begin{definition}
A state of a diagram of a knotoid, $K$, consists in a choice of local state for each crossing of $K$. 
\end{definition}

\begin{definition}
The bracket polynomial of a knotoid diagram $K$ is defined as:

\begin{equation}\label{bracketformula}
\langle K\rangle =\sum_SA^{\sigma(S)}d^{||S||-1}
\end{equation}

\noindent where the sum is taken over all states, $\sigma(S)$ is the sum of the labels of the state $S$, $||S||$ is the number of components of $S$, and $d=(-A^2-A^{-2})$.
\end{definition}

\begin{remark}
The classical bracket polynomial of knots is defined using formula \ref{bracketformula}, with the same Skein relations as in Eq. \ref{skein}, except the last one, where an arc is replaced by a circle.
The classical bracket polynomial is not a topological invariant for knots (it is not invariant under the Reidemeister 1 move) and depends on the knot diagram used for its computation. Similarly, the bracket polynomial of knotoids is not invariant under $\Omega_1$ (the Reidemeister 1 move) and depends on the knotoid diagram.
\end{remark}

\subsubsection{The Jones polynomial of knotoids}

The Jones polynomial of knotoids is an invariant of knotoids and many component knotoids, called multiknotoids or linkoids,  (equivalent knotoids/linkoids map to the same polynomial) and can be defined using the normalized bracket polynomial. 
The normalized the bracket polynomial is defined as follows:

\begin{equation}
f_K=(-A^{-3})^{-wr(K)}\langle K\rangle
\end{equation}

\noindent where $wr(K)$ is the writhe of the knotoid diagram $K$. 

The normalized bracket polynomial of knotoids in $S^2$ generalizes the Jones polynomial of knotoids with the substitution $A=t^{-1/4}$.

\begin{remark} The same definition, where $K$ is a knot diagram, applies to simple closed curves to give the Jones polynomial of knots and links, a topological invariant of knots and links.
\end{remark}

\section{The bracket polynomial of a curve in 3-space}\label{bracketopen}

Consider an open or closed curve in 3-space (we will also call it chain).
A (generic) projection of a curve (fixed in 3-space) can give a different knotoid diagram (or knot diagram), depending on the choice of projection direction. We define the bracket polynomial of a 3-dimensional curve as the \textit{average} of the Kauffman bracket polynomial of a projection of the curve over all possible projection directions.  The definition is made precise as follows:

\begin{definition}
Let $l$ denote a curve in 3-space. Let $(l)_{\vec{\xi}}$ denote the projection of $l$ on a plane with normal vector $\xi$. The bracket polynomial of $l$ is defined as: 

\begin{equation}\label{avk}
\langle l\rangle=\frac{1}{4\pi}\int_{\vec{\xi}\in S^2}\langle K((l)_{\vec{\xi}})\rangle dS
\end{equation}

\noindent where the integral is over all vectors in $S^2$ except a set of measure zero (corresponding to non-generic projections).

\end{definition}

\begin{remark}
The same definition applies to define the bracket polynomial of many open and/or closed curves in space by replacing $l$ by a many component open/closed or mixed collection of open and closed curves. In this manuscript we focus on one component, but the same analysis holds for many chains in 3-space.
\end{remark}

\noindent\textit{Properties of the bracket polynomial of chains in 3-space}

\noindent (i) The bracket polynomial does not depend on any particular projection of the chain (open or closed).

\noindent(ii) For an open chain this polynomial is not the polynomial of a corresponding/approximating closed curve, nor that of a corresponding/approximating knotoid.

\noindent(iii) For both open and closed chains, the bracket polynomial has real coefficients.

\noindent(iv) The bracket polynomial defined in Eq. \ref{avk} is not a topological invariant, but it is a continuous function of the chain coordinates for both open and closed chains (see Corollary \ref{contcurve}).

In the following, we will show that the bracket polynomial of curves in 3-space attains a simpler expression for polygonal chains. However, similar arguments can be used to extend this simpler expression to any curve in 3-space (polygonal or not).

Let $EW_n$ denote the space of configurations of polygonal chains of $n$ edges. Let $E_n$ denote a polygonal chain of $n$ edges in 3-space. Then only a finite number of different knotoid (or knot) types can occur in any projection of $E_n$. Let $k(n)$ be the total number of knotoids that can be realized by a projection of a 3-dimensional polygonal chain with $n$ edges, we denote $K_i, i=1,\dotsc k(n)$. 

Then Eq. \ref{avk} is equivalent to the following sum:

\begin{equation}\label{avk2}
\langle E_n\rangle=\sum_{i=1}^kP(K(E_n)_{\vec{\xi}}=K_i)\langle K_i\rangle=\sum_{i=1}^kp^{(n)}_i\langle K_i\rangle
\end{equation}

\noindent where $K((E_n)_{\vec{\xi}})$ denotes the knotoid corresponding to $(E_n)_{\vec{\xi}}$ and we denote $p^{(n)}_i=P(K(E_n)_{\vec{\xi}}=K_i)$, the probability that a projection of $E_n$ gives knotoid $K_i$.

\begin{remark}
Here and in the following, by ``probability that a projection of $x$ gives $y$'' we mean the ratio of the area on the (unit) sphere that defines vectors with respect to which the projection of $x$ is $y$ (or of type $y$) over the area of the entire sphere. 
\end{remark}

Let $m$ denote the maximum degree of $\langle K_i\rangle,i=1,\dotsc,k$ and let $L_m$ denote the space of Laurent polynomials of degree less than or equal to $m$. Then $\langle E_n\rangle$ is a function from $EW_n$ to $L_m$.

\begin{lemma}\label{prob} The probability $p^{(n)}_i=P(K(E_n)_{\vec{\xi}}=K_i)$ is a continuous function of the chain coordinates of $E_n$.
\end{lemma}

\begin{proof}
Notice that 

\begin{equation}
p^{(n)}_i=\frac{2A_0 }{4\pi}
\end{equation}

\noindent where $A_0=$ Area on the sphere corresponding to vectors $\vec{\xi}$ such that: $K((E_n)_{\vec{\xi}})=K_i$. For a polygonal chain, this area will be bounded by a finite number of great circles, each of which is determined by an edge and a vertex of the polygonal chain, as in \cite{Banchoff1976}.

Let $\epsilon>0$. Let $\vec{a}_j$ be the position of a vertex of $E_n$. Let $d=\min_{k,l}d_{k,l}$, where $d_{k,l}=dist(\vec{a}_j,\vec{a}_k-\vec{a}_l)$ (the distance between the vertex $\vec{a}_j$ and the segment connecting $\vec{a}_k,\vec{a}_l$) . Suppose that $\vec{a}_j$ changes by $\delta \vec{a}$, such that $||\delta \vec{a}||<\frac{2\pi d\epsilon}{8(n-2)}$. Then, the projection of the edges $e_{j-1}=\vec{a}_j-\vec{a}_{j-1}$ and $e_j\vec{a}_{j+1}-\vec{a}_j$ in any projection direction might change and the great circles involving the vertex $\vec{a}_j$ might change as well. Each of these two edges, $e_{j-1},e_j$ is involved in $(n-2)$ pairs of edges with which they may cross in a projection  and each such pair consists of 3 faces containing $a_j$, one of which is counted in both the $e_{j-1}$ and the $e_j$ pairs. Thus, a change in $a_j$ can affect $4(n-2)$ planes. Let $\vec{u}$ be the normal vector to one of these planes, say the one formed by the vertices $\vec{a}_j,\vec{a}_l,\vec{a}_{l+1}$. The normal vector to the new plane containing $\vec{a}_j+\delta\vec{a},\vec{a}_l,\vec{a}_{l+1}$, will change to $\vec{u}+\vec{\delta u}$. If that plane was one of the great circles bounding $A_0$, then $A_0$ may also change to $A_0'$ (and $p_i^{(n)}$ to $p_i^{(n)}\prime$, accordingly). The change in area $|A_0-A_0'|$ will be bounded above by the area of the lune defined by the great circles with normal vectors $\vec{u}$ and $\vec{u}+\vec{\delta u}$, which is equal to $\alpha=2\theta$, where $\theta$ is the dihedral angle between the two great circles, which is equal to the angle between $\vec{u}$ and $\vec{u}+\vec{\delta u}$. The maximum value of that angle will occur if $\delta \vec{a}$ is orthogonal to the plane $\vec{a}_j,\vec{a}_l,\vec{a}_{l+1}$, which means when $\delta\vec{a}$ is parallel to $\vec{u}$. Then the angle $\theta$ is that of a right triangle with one edge of length $d_{k,l}=dist(\vec{a}_j,\vec{a}_k-\vec{a}_l)$ and the other of length $||\delta \vec{a}||$. Thus $\tan\theta=\frac{||\delta \vec{a}||}{d_{k,l}}$.  Thus, the change in the area is

\begin{align}
|A_0-A_0'|&\leq4(n-2) 2\arctan(\frac{||\delta \vec{a}||}{d})<8(n-2)\arctan(\frac{2\pi d\epsilon}{8(n-2)d})\nonumber\\
&\approx 8(n-2)\frac{\epsilon}{8(n-2)}=2\pi\epsilon
\end{align}

\noindent where we used the small angle approximation. Thus $|p_i^{(n)}-p_i^{(n)}\prime|<\epsilon$.

\end{proof}

\begin{proposition}\label{cont}
The bracket polynomial, $\langle E_n\rangle$, is a continuous function of the chain coordinates. In other words it is a continuous function in the space of configurations of $E_n$.
\end{proposition}

\begin{proof}
We consider the standard Euclidean norm over the space of Laurent polynomials of a fixed degree. Since the coefficients of this polynomial are $p_i$, then $||\langle E_n\rangle||=\sqrt{\sum p_i^2}$. Since each coefficient $p_i$ is a continuous function of the chain coordinates, it follows, that $\langle K(E_n)\rangle$ will also be continuous with the norm mentioned above.
\end{proof}

\begin{corollary}\label{contcurve}
The bracket polynomial of a curve $l$ in space,  $\langle l\rangle$, is a continuous function in the space of configurations of $l$.
\end{corollary}

\begin{proof}
By approximating $l$ by a polygonal curve, $l_n$ and taking the limit as $n\rightarrow\infty$ by Proposition \ref{cont}, follows that $l$ is continuous.
\end{proof}

\begin{remark}
The above definitions hold if one considers knotoids in $S^2$ or planar knotoids. The difference will be in the number $k(n)$, which is higher for planar knotoids.
\end{remark}

\begin{remark}
Using the state formula for the bracket polynomial of a knotoid, we obtain the following state formula for the bracket polynomial of a polygonal curve in 3-space:

\begin{equation}\label{state}
\langle E_n\rangle=\sum_{i=1}^kP(K(E_n)_{\vec{\xi}}=K_i)\sum_{j=1}^{m_i}A^{\sigma(S_j)}d^{||S_j||-1}
\end{equation}

\noindent where the first sum is taken over all realizable knotoids of $n$ edges and the second sum is taken over all states, $S_j$, of the $i$-th realizable knotoid, $\sigma(S_j)$ is the sum of the labels of the state $S_j$, $||S_j||$ is the number of components of $S_j$, and $d=(-A^2-A^{-2})$.

By expanding the summands, $\langle K(E_n)_{\vec{\xi}}\rangle$ can be expressed as

\begin{equation}\label{stat2}
\langle E_n\rangle=\sum_{l=1}^{M}p^{(n)}_lA^{\sigma(S_l)}d^{||S_l||-1}
\end{equation}

\noindent where $S_l$ are all the possible states of $E_n$, $M$ is the total number of distinct states that appear as projections of $E_n$ and $p^{(n)}_l$ is equal to the probability of state $l$. Using the standard definition of a state of a diagram of a knotoid, the states are uniquely identified for a knotoid diagram, giving $S_j\neq S_{j\prime}$ for any two states of a knotoid $K_i$ and also $S_j\neq S_u$ for any states $S_j$ of $K_i$ and $S_u$ of $K_v$. Then $M=\sum_{i=1}^km_i$, where $m_i$ are the states corresponding to the knotoid diagram $K_i$, and $p^{(n)}_l$ is equal to the probability of obtaining a specific diagram of the knotoid to which $S_l$ corresponds. Different definitions of state or of a probability of a state can be used, changing the expression of Eq. \ref{stat2}. For example, we could define the probability of a state $S_l$ in the space of configurations of states to be $P(S_l) =\frac{p^{(n)}_l}{2^n}$, where $p^{(n)}_l$ is equal to the probability of obtaining the specific knotoid diagram to which $S_l$ corresponds. Then Eq. \ref{stat2} would become:

\begin{equation}\label{stat3}
\langle E_n\rangle=\sum_{l=1}^{M}2^nP(S_l)A^{\sigma(S_l)}d^{||S_l||-1}
\end{equation}

\end{remark}

\subsection{The Jones polynomial of open chains in 3-space}

The Jones polynomial of an open chain in 3-space is defined using the normalized bracket polynomial of an open chain:

\begin{definition}
Let $l$ denote a curve in 3-space. Let $(l)_{\vec{\xi}}$ denote the projection of $l$ on a plane with normal vector $\xi$. 

The normalized bracket polynomial of $l$ is defined as: 

\begin{equation}\label{avnk}
f_{K(l)}=\frac{1}{4\pi}\int_{\vec{\xi}\in S^2}(-A^3)^{-wr((l)_{\vec{\xi}})}\langle (l)_{\vec{\xi}}\rangle dS
\end{equation}

\noindent where the integral is over all vectors in $S^2$ except a set of measure zero (corresponding to non-generic projections).

\end{definition}

\begin{remark}
The same definition applies to define the Jones polynomial of many open and/or closed curves in space by replacing $l$ by a many component open/closed or mixed collection of open and closed curves. In the case of a collection of closed chains (a traditional link), the Jones polynomial is a topological invariant. In the case of open chains it is a continuous function in the space of configurations. In this manuscript we focus on one component, but the same analysis holds for many chains in 3-space.
\end{remark}

\noindent\textit{Properties of the Jones polynomial of chains in 3-space}

\noindent(i) For closed chains, the Jones polynomial defined in Eq. \ref{avnk} is  a topological invariant and coincides with the classical Jones polynomial of a knot (see Corollary \ref{Jonesclosed})

\noindent(ii) For open chains, the Jones polynomial has real coefficients and is a continuous function of the chain coordinates   (see Corollary \ref{contcurveJ}).

\noindent(iii) For an open chain the Jones polynomial is not the polynomial of a corresponding/approximating closed curve, nor that of a corresponding/approximating knotoid.

\begin{corollary}\label{Jonesclosed}
In the case where $l$ is a closed curve, $f_{l}=(-A^3)^{-wr((l)_{\vec{\xi}})}\langle (l)_{\vec{\xi}}\rangle$ for all $\vec{\xi}\in S^2$. 
\end{corollary}
 
\begin{proof}
Le $l$ be a closed curve, and let $\vec{\xi}\in S^2$. Then its projection $l_{\vec{\xi}}$ is a knot diagram and $(-A^3)^{-wr((l)_{\vec{\xi}})}\langle (l)_{\vec{\xi}}$ is a topological invariant that does not depend on the particular diagram of the knot. Thus $f_{l}=\frac{1}{4\pi}\int_{\vec{\xi}\in S^2}(-A^3)^{-wr((l)_{\vec{\xi}})}\langle (l)_{\vec{\xi}}\rangle dS=\frac{1}{4\pi}4\pi(-A^3)^{-wr((l)_{\vec{\xi}})}\langle (l)_{\vec{\xi}}\rangle$.
\end{proof}

For a polygonal chain of $n$ edges,  Eq. \ref{avnk} is equivalent to the following sum:

\begin{equation}\label{avnk2}
f_{(E_n)}=\sum_{i=1}^k\sum_{j=-m}^mP(K(E_n)_{\vec{\xi}}=K_i,wr((E_n)_{\vec{\xi}}=j))(-A^3)^{-j}\langle K_{i,j}\rangle=\sum_{i=1}^k\sum_{j=-m}^mp^{(n)}_{i,j}(-A^{3})^{-j}\langle K_{i,j}\rangle
\end{equation}

\noindent where we denote $p^{(n)}_{i,j}=P(K(E_n)_{\vec{\xi}}=K_i,wr((E_n)_{\vec{\xi}})=j)$, $k=k(n)$ and  $m=m(n,i)$.

\begin{corollary}\label{contcurveJ} The normalized bracket polynomial of an open chain in 3-space is a continuous function of the chain coordinates.
\end{corollary}

\begin{proof} In a similar way as in Lemma \ref{prob}, one can show that for a polygonal chain of $n$ edges, $p^{(n)}_{i,j}$ is a continuous function of the chain coordinates for all $i,j,n$ and use that for the limiting case of any simple curve $l$ in 3-space.
\end{proof}

\noindent\textbf{Example 1:}  Figure \ref{fig:examples} shows three snapshots of a polygonal chain, $I$, whose last edge deforms with time as the last vertex position changes according to the parametrization $I(t)=((0,1,0),(0,0,0),(-0.2,0.8,0.8),(0.1,0.8,-0.8),(0.1+1.2\cos (a+t),0.5,-0.8+1.2\sin(a+t)))$. The coordinates of the chain in the three snapshots in Figure \ref{fig:examples} are:

$I(t_0)=((0, 1, 0),(0, 0, 0),(-0.2, 0.8, 0.8),(0.1, 0.8, -0.8),(0.76, 0.5, 0.19))$, 

$I(t_1)=((0, 1, 0),(0, 0, 0),(-0.2, 0.8, 0.8),(0.1, 0.8, -0.8),(0.35, 0.5, 0.37))$ and 

$I(t_2)=((0, 1, 0),(0, 0, 0),(-0.2, 0.8, 0.8),(0.1, 0.8, -0.8), (-0.02, 0.5, 0.39))$.

\noindent where $t_0=0, t_1=4000$ and  $t_2=11300$, in units of $2\pi/100000$, and $a=32000\pi/100000$.


\begin{figure}[H]
   \begin{center}
     \includegraphics[width=0.3\textwidth]{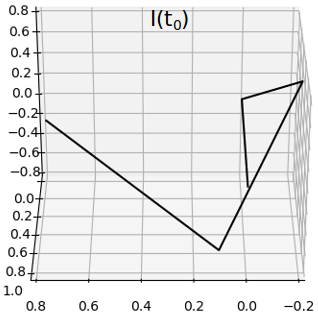}\includegraphics[width=0.3\textwidth]{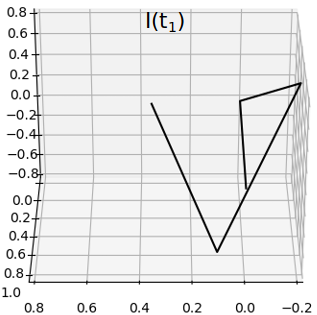}\includegraphics[width=0.3\textwidth]{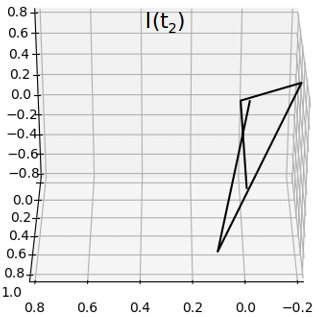}\\
     \caption{Three snapshots of a polygonal chain in 3-space with 3 fixed edges and one deforming edge in 3-space. From $t_0$ to $t_2$, the chain tightens a configuration that gives the knotoid $k2.1$ in most projection directions and could lead to the creation of a trefoil knot if was able to thread through (more edges are needed for that \cite{Calvo2001}).}
     \label{fig:examples}
   \end{center}
\end{figure}

The Kauffman bracket at the start and end time is:

\begin{equation}
\begin{split}
    &\langle I(t_0)\rangle= 0.06A^2-0.77A^{-3}-0.06A^{-4}+0.07A^{-6}+0.15\\
    &\langle I(t_2)\rangle=0.71A^2-0.71A^{-4}-0.14A^{-3}+0.05A^{-6}+0.14\\
\end{split}
\end{equation}

\noindent to be compared with the values of the bracket polynomial of the typical configuration of the right-handed trefoil knot, $T_R$ and the right handed $k2.1$ knotoid, which are equal to

\begin{equation}
\begin{split}
    &\langle T_R\rangle= A^{-7}+A^{5}-A^{-3}\\
    &\langle k2.1\rangle= A^{2}-A^{-4}+1
\end{split}
\end{equation}

The Jones polynomial at each time is

\begin{equation}
\begin{split}
    &f(I(t_0))= 0.06t-0.06t^{5/2}+0.06t^{3/2}+0.94\\
    &f(I(t_2))= 0.71t-0.71t^{5/2}+0.71t^{3/2}+0.29\\
\end{split}
\end{equation}

to be compared with the Jones polynomial of the right-handed trefoil knot, $T_R$ and the right handed $k2.1$ knotoid, which are equal to

\begin{equation}
\begin{split}
    &f(T_R)= t+t^3-t^4\\
    &f(k2.1)= t+t^{3/2}-t^{5/2}
\end{split}
\end{equation}

\begin{remark}
Using the state formula for the bracket polynomial of a knotoid, we obtain the following state formula for the normalized bracket polynomial of a polygonal curve in 3-space:

\begin{equation}\label{avnk3}
f_{(E_n)}=\sum_{i=1}^k\sum_{j=1}^mp^{(n)}_{i,j}(-A^{3})^{-j}\sum_{S_i}A^{\sigma(S_i)}d^{||S_i||-1}
\end{equation}

\noindent where the first sum is taken over all realizable knotoids of $n$ edges and the second sum is taken over all states, $\sigma(S_i)$ of the $i$-th realizable knotoid, is the sum of the labels of the state $S_i$, $||S_i||$ is the number of components of $S_i$, and $d=(-A^2-A^{-2})$.

\end{remark}

\begin{remark}[Comparison with previous methods]
Due to the urgency of measuring complexity in physical systems, several approaches have appeared in the last decade that attempt to use  knot and link polynomials \cite{Laso2009,Millett2004,Sulkowska2012,Gugumcu2017,Goundaroulis2017}. The underlying idea in these methods is to approximate an open chain in 3-space by a knot (\textit{dominant knot}) or by a knotoid (\textit{dominant knotoid}) that best captures its entanglement. Both the dominant knot and the dominant knotoid have been successful in characterizing proteins \cite{Sulkowska2012,Goundaroulis2017}. 
Even though these approaches are very helpful, they can at best approximate an open chain by either one closed chain or by one of its projections, respectively, and in practice, they might even give different answers for different choice of points on the sphere. 
Putting these methods in the framework we established in this paper, they consist in computing the knot-type or the knotoid type with highest probability of occurring in a projection. In this study instead, we use the average of all the bracket polynomials of all the knotoids that occur. As we discussed in the previous paragraphs, this simple modification provides for the first time a well defined measure of entanglement of open chains, other than the Gauss linking integral (see all the properties mentioned above). To understand the difference between the information captured by the two methods we draw a comparison between the linking number and the Gauss linking integral: the dominant knot/knotoid method would correspond to the integer linking number that occurs in the most projections of an open chain, while the definition we give here, would correspond to that of the Gauss linking integral (the average linking number over all projections).
\end{remark}

\begin{remark}[Comparison of the Jones polynomial of open chains with the Gauss linking integral]
Notice that when applied to one chain, the Jones polynomial gives stronger information that the writhe of the chain. It is known that one can create an open or closed unknotted chain which has high writhe (consider for example a helix). However, it is impossible to create a Jones polynomial of an open or closed chain that does not contain a knot, which is the same as that of an open or closed chain that contains a knot. The Jones polynomial of open or closed chains is stronger than the pairwise linking number of open or closed chains. For example, the Gauss linking integral cannot detect open or closed conformations of the Borromean ring. However, the Jones polynomial of the Borromean ring is different from that of the unlink and, by continuity, this is the case also for an open Borromean ring.
\end{remark}

\section{A finite form for the bracket polynomial of a polygonal curve with 4 edges}\label{finite}

In this section we show that an equivalent finite form of bracket polynomial exists, reducing the computation of the integral to a computation of a few dot and cross products between vectors and some arcsin evaluations. Here we provide a finite form of the bracket polynomial for a polygonal chain of 4 edges. This could lead to the creation of its finite form for more edges.

\subsection{Closed chains}\label{bracketclosed}

The first non-trivial bracket polynomial of a closed chain is that of a polygon of 4 edges, since a polygon of 3 edges is a triangle in 3-space and all projections give a diagram of no crossings except a set of measure zero which corresponds to non-generic projections.
Let $P_4$ denote a polygon of 4 edges, $e_1,e_2,e_3,e_4$ that connect the vertices $(0,1),(1,2),(2,3)$ and $(3,0)$, respectively.  Let $\epsilon_{i,j}$ denote the sign of the crossing between the projections of the edges $e_i,e_j$ when they cross. Notice that $\epsilon_{i,j}$ is independent of the projection direction and can take the values 1 and -1.

\begin{proposition}
The bracket polynomial of a polygon of 4 edges, $e_1,e_2,e_3,e_4$, in 3-space, $P_4$, is equal to:

\begin{equation}\label{bp4}
\langle P_4\rangle=2|L(e_1,e_3)|(-A^{3\epsilon_{1,3}})+2|L(e_2,e_4)|(-A^{3\epsilon_{2,4}})+(1-ACN(P_4))
\end{equation}

\noindent where $L$ denotes the Gauss linking integral and $ACN$ denotes the average crossing number.

\end{proposition}

\begin{proof}

In any projection direction there are 3 possible diagrams that may occur as a projection of $P_4$:  a diagram with no crossing, or a crossing between the projections of $e_1,e_3$ or a crossing between the projections of $e_2,e_4$. Notice that not both crossings at the same diagram are possible (the line defined by the projection of $e_1$ cuts the plane in two regions. Since the projection of $e_3$ intersects the projection of $e_1$, the projections of the vertices 2 and 3 lie in different regions. Since $e_2$ joins vertex 1 with 2 and $e_4$ joins vertex 3 with 0, $e_2,e_4$ lie in different regions, thus they cannot cross.) In the case where there is no crossing, the bracket polynomial of that projection is equal to 1. When there is a crossing, the bracket polynomial is equal to $-A^{\pm3}$, where the sign of the exponent is determined by the sign of the crossing in the projection.
 Since the probability of $e_2,e_4$ crossing is equal to $2|L(e_2,e_4)|$ and the probability of $e_1,e_3$ crossing is $2|L(e_1,e_3)|$, then the bracket polynomial is

\begin{equation}\label{bp41}
\langle P_4\rangle=2|L(e_1,e_3)|(-A^{3\epsilon_{1,3}})+2|L(e_2,e_4)|(-A^{3\epsilon_{2,4}})+(1-ACN(P_4))
\end{equation}

\noindent where we used the fact that $ACN(P_4)=2|L(e_1,e_3)|+2|L(e_2,e_4)|$.
Notice that, due to the connectivity of the chain, $\epsilon_{1,3}=-\epsilon_{2,4}$, thus Eq. \ref{bp41} could be expressed as

\begin{align}
\langle P_4\rangle&=2|L(e_1,e_3)|(-A^{3\epsilon_{1,3}})+2|L(e_2,e_4)|(-A^{-3\epsilon_{1,3}})+(1-ACN(P_4))\nonumber
\end{align}

\end{proof}

\subsection{Open chains}

In the case of a polygonal chain with 3 edges, we denote $E_3$, the Kauffman bracket polynomial is always trivial, but the writhe of a diagram of a projection of $E_3$ can be 0 or $\pm1$, depending on whether $e_1,e_3$ cross when projected in a direction $\vec{\xi}$. 

\begin{proposition}
Let $E_3$ denote a polygonal chain of 3 edges, $e_1,e_2,e_3$ in 3-space, then the bracket polynomial of $E_3$ is

$$
\langle E_3\rangle=2|L(e_1,e_3)|(-A^3)^{\epsilon_{13}}+(1-2|L(e_1,e_3)|)
$$

\noindent where $\epsilon_{1,3}$ is the sign of $L(e_1,e_3)$
\end{proposition}

\begin{proof} Consider a polygonal chain of 3 edges $e_1,e_2,e_3$, $(E_3)$. Then in a projection of $E_3$, $(E_3)_{\xi}$, one either sees no crossings, so $\langle (E_3)_{\xi}\rangle=1$, or there is a crossing between $e_1$ and $e_3$, in which case $\langle (E_3)_{\xi}\rangle=-A^{\epsilon_{1,3}}$, thus

\begin{equation}
\begin{split}
&\langle E_3\rangle= P(K((E_3)_{\xi})=k0,wr((E_3)_{\xi})=0)+P(K((E_3)_{\xi})=k0,wr((E_4)_{\xi})=\epsilon_{1,3})(-A^3)^{\epsilon_{1,3}}\nonumber\\
&=(1-2|L(e_1,e_3)|)+2|L(e_1,e_3)|(-A^3)^{\epsilon_{1,3}}\nonumber
\end{split}
\end{equation}

\end{proof}

Let $E_4$ be composed by 4 edges, $e_1,e_2.e_3,e_4$, connecting the vertices $(0,1),(1,2),(2,3),(3,4)$, respectively.

\begin{proposition}\label{number4} $k(4)=2$ (There are only two different knotoids that can be realized by a 3-dimensional polygonal chain with 4 edges). 
\end{proposition}

\begin{proof} In a projection of $E_4$, crossings may occur only between the projections of the pairs of edges: $e_1,e_3$, $e_1,e_4$ and $e_2,e_4$. Therefore we have the following 5 possible combinations (see Figure \ref{fig:4knotoids}):

\noindent case A: one crossing, between the projections of:  $e_1,e_3$ or  $e_1,e_4$ or   $e_2,e_4$ 

\noindent case B: two crossings, between the projections of: $e_1,e_3$ and $e_1,e_4$ (giving two possible diagrams, $i$ and $i\prime$) or  $e_2,e_4$ and $e_1,e_4$ (giving two possible diagrams, $ii$ and $ii\prime$) or  $e_1,e_3$ and $e_2,e_4$ (not realizable, see below)

\noindent case C: three crossings, between the projections of: $e_1,e_3$ and $e_1,e_4$ and $e_2,e_4$.

The case B with $e_1,e_3$ and $e_2,e_4$ crossings is not realizable: The projection of $e_1$ defines a line in the plane that divides it in two regions. Suppose that the projection of $e_3$ intersects $e_1$. Then the endpoints of $e_3$ lie in opposite regions and are the endpoint and the starting point of $e_2$ and $e_4$, respectively. Thus the starting point of $e_4$ is in the opposite region of the one where $e_2$ lies in and to intersect $e_2$ it must also intersect $e_1$.

\begin{figure}[H]
   \begin{center}
     \includegraphics[width=0.8\textwidth]{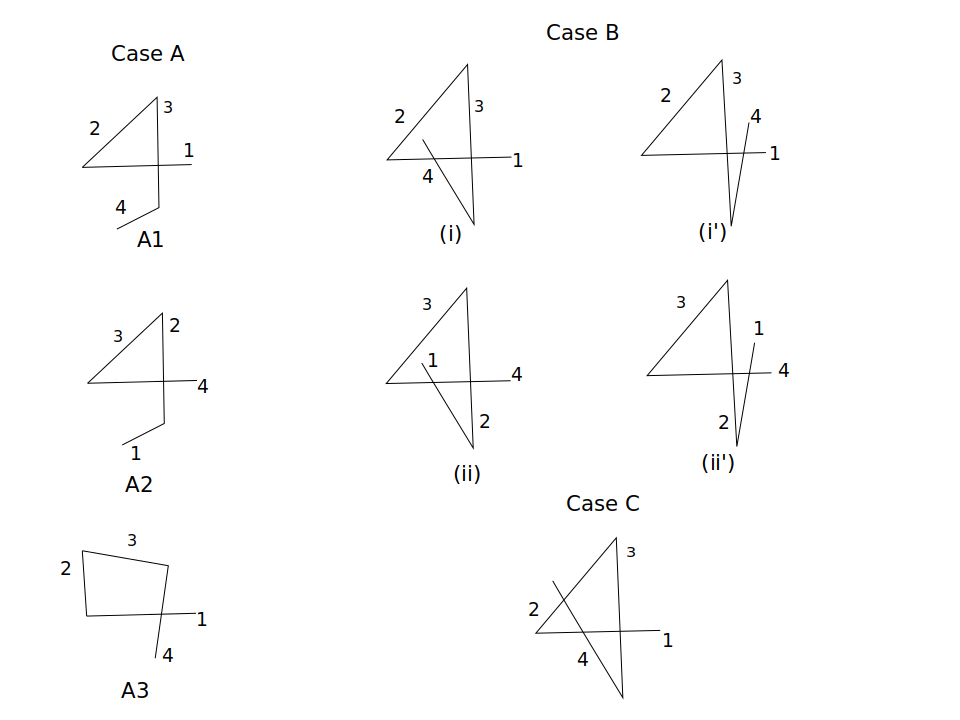}
     \caption{Possible diagrams of a projection of a polygonal curve with 4 edges, $e_1,e_2,e_3,e_4$. Each crossing may be over or under, except for case C, where constraints apply due to the chain rigidity (see proof of Proposition \ref{number4}). Only case B (i) and (ii) can give a non-trivial knotoid when both crossings have the same sign.  Therefore, if $(E_4)_{\xi}$ is non-trivial, it is of only one type: $k2.1$.  }
     \label{fig:4knotoids}
   \end{center}
\end{figure}

Figure \ref{fig:4knotoids} shows the different cases of diagrams with undefined over or under crossings which give rise to realizable knotoids. The diagrams of case A and Case B (i') and (ii') are all trivial and case C is realizable only when it is trivial. The diagrams of case B (i) and (ii) are non-trivial (in $S^2$) only when the crossings between the involved edges have the same sign, ie. $\epsilon_{1,3}=\epsilon_{1,4}$ or $\epsilon_{1,4}=\epsilon_{2,4}$, resp., in which case, they both represent the knotoid $k2.1$ \cite{Gugumcu2017}.

\end{proof}

The next proposition shows that when the projection of $E_4$ is of type $k2.1$, it can be only one of the two possible k2.1 diagrams (case B (i) or (ii)) in any projection direction.

\begin{proposition}\label{case}
Let $E_4$ denote a polygonal chain of 4 edges in 3 space. If there is $\vec{\xi}_1$ such that $(E_4)_{\vec{\xi}_1}=$ case $B (i)$, then there does not exist $\vec{\xi}\in S^2$, $\vec{\xi}\neq\vec{\xi}_1$ such that $(E_4)_{\vec{\xi}}=B(ii)$ (and vice-versa). 
\end{proposition}

\begin{proof}

Without loss of generality, suppose that $\epsilon_{1,3}=\epsilon_{1,4}=1$ and that there exists $\xi_1$ such that $(E_4)_{\xi_1}$ is of the form (i). Then, $(e_3\times e_4)\cdot e_1>0$. Suppose that there is $\xi_1$ such that $(E_4)_{\xi_1}$ is of the form (ii). Then $(e_3\times e_4)\cdot e_1<0$, contradiction.
\end{proof}

Let $E_4$ denote a polygonal chain of 4 edges. Then, by Propositions \ref{number4} and \ref{case}, the only non-trivial bracket polynomial is $k2.1$ and the writhe of the diagram is either 2 or -2. All the possible writhe values in a $k0$ (trivial knotoid) diagram of $E_4$ can be determined by inspection of  all the possible diagrams of a chain of 4 edges, given in Proposition \ref{number4}. Let us denote these diagrams as $k0_{A_1},k0_{A_2},k0_{A_3},k0_{B_i},k0_{B_{i\prime}},k0_{B_{ii}},k0_{B_{ii\prime}},k0_C$. 
Let us denote by $wr$ the writhe of a diagram. Then one can see that $wr(k0_{A_1})=\pm1,wr(k0_{A_2})=\pm1,wr(k0_{A_3})=\pm1,wr(k0_{B_i})=0$ or $=\pm2$,$wr(k0_{B_{i\prime}})=0$ or $\pm2$, $wr(k0_{B_{ii}})=0$ or $\pm2$, $wr(k0_{B_{ii\prime}})=0$ or $\pm2$,$wr(k0_C)=\pm1$. Thus the bracket polynomial of $E_4$ has the following form:

\begin{equation}
    \begin{split}
&\langle E_4\rangle= P(K(E_4)_{\xi})=k2.1)\langle k2.1\rangle+\sum_{j=-2}^2P(K((E_4)_{\xi})=k0,wr((E_4)_{\xi})=j)(-A^3)^{j}\\ 
&= P(K((E_4)_{\xi})=k2.1)(A^2-A^{-4}+1)+\sum_{j=-2}^2P(K((E_4)_{\xi})=k0,wr((E_4)_{\xi})=j)(-A^3)^{j}\nonumber
\end{split}
\end{equation}

\noindent where  $P(K((E_4)_{\xi})=k2.1)$ denotes the geometruc probability that a projection of $E_4$ gives the non-trivial knotoid $k2.1$ and where $P(K((E_4)_{\xi})=k0,wr((E_4)_{\xi})=j)$ denotes the probability of obtaining a diagram of the trivial knotoid with writhe $j$. 

The rest of this section is focused on obtaining finite forms for these probabilities. More precisely, a finite form for $P(K((E_4)_{\xi})=k2.1)$ is derived in Theorem \ref{knotoid4} and a finite form for all $P(K((E_4)_{\xi})=k0,wr((E_4)_{\xi})=j)$ is derived in Theorem \ref{fullbracket}.

In the following definition we gather some of the notation used so far, together with some new definitions, necessary for the rest of the manuscript.

\begin{definition}\label{defn1}
Throughout this manuscript, we will denote by $Q_{i,j}$ the spherical polygon which corresponds to projections where the edges $e_i,e_j$ cross. $Q_{i,j}^A$ is the antipodal of $Q_{i,j}$ on the sphere. $Q_{i,j,k}$ is the  spherical polygon which corresponds to projections where the edges $e_i,e_j$ and $e_i,e_k$ cross, it is equal ro $Q_{i,j,k}=(Q_{i,j}\cap Q_{i,k})\cup(Q_{i,j}^A\cap Q_{i,k})$. $Q_{i,j,k}^A$ is the antipodal of $Q_{i,j,k}$ on the sphere. We denote $(\vec{w}_1,\dotsc,\vec{w}_k)$ the spherical polygon formed by the intersection of great circles with normal vectors $\vec{w}_1,\dotsc,\vec{w}_k$ in the counterclockwise orientation. $A(Q_{i,j})$, $A(Q_{i,j,k})$ and  $A(\vec{w}_1,\dotsc,\vec{w}_k)$  denote the area of $Q_{i,j}$, the area of $Q_{i,j,k}$ and the area of $(\vec{w}_1,\dotsc,\vec{w}_k)$, respectively. We denote by $T_{i,j}$, the quadrilateral in 3-space that is formed by joining the vertices of the edge $e_i$ with the vertices of the edge $e_j$.  The normal vectors of $T_{i,j}$, denoted $\vec{n}_1,\vec{n}_2,\vec{n}_3,\vec{n}_4$, are normal vectors to the great circles that bound $Q_{i,j}$ and are determined by the algorithm described in Section \ref{Gaussfinite} when $i<j$. We define  the spherical faces of the quadrangles from the quadrilateral as follows: at each vertex of the quadrilateral extend each edge by length 1 and connect those segments that share a common vertex by an arc on the unit sphere (see Figure \ref{fig:qk21} for an illustrative example).  We call the spherical faces at the vertex $i$ and $i+1$, (corresponding to the vectors $\vec{n}_1,\vec{n}_3$), the left and right faces of $T_{i,j}$ and the spherical faces at $j$ and $j+1$ (corresponding to $\vec{n}_2$ and $\vec{n}_4$), the top and bottom faces. One pair bounds $Q_{i,j}$ and the other bounds $Q_{i,j}^A$, but the reflections of these spherical faces through the center of the sphere create both quadrangles. We will say that $T_{i,j}$ \textit{generates} the quadrangle that contains the pair of right and left spherical faces of $T_{i,j}$ (the spherical faces at $i$ and $i+1$, respectively). We notice that in a quadrangle generated by a quadrilateral $T_{i,j}$ the vectors either point inward or outward the quadrangle and their numbering either follows a counterclockwise or clockwise orientation on $Q_{i,j}$, depending on the sign of $\epsilon_{i,j}$. If the normal vectors of $Q_{i,j}$ point inwards (outwards resp.) then those of $Q_{i,j}^A$ point outwards (inwards resp.) and with the opposite numbering sequence (clockwise/counterclockwise). We call the antipodal quadrilateral of $T_{i,j}$, we denote $T_{i,j}^A$, the quadrilateral which generates $Q_{i,j}^A$. We denote its normal vectors as $\vec{n}_1^A, \vec{n}_2^A,\vec{n}_3^A,\vec{n}_4^A$

\end{definition}

\begin{lemma}\label{antipodal}
Let $T_{i,j}$ denote the quadrilateral formed by $e_i,e_j$ with vertices at the points $\vec{p}_i,\vec{p}_{i+1},\vec{p}_j,\vec{p}_{j+1}$. The antipodal of $T_{i,j}$, $T_{i,j}^A$, is the tetrahedral formed by the edge $e_i$ and the edge $e_j^A$, with vertices $\vec{p}_{j^A}=\vec{p}_{i+1}-(\vec{p}_{j}-\vec{p}_i)$ and $\vec{p}_{(j+1)^A}=\vec{p}_{i+1}-(\vec{p}_{j+1}-\vec{p}_i)$.
\end{lemma}

\begin{proof}
Let $\vec{n}_1,\vec{n}_2,\vec{n}_3,\vec{n}_4$ denote the normal vectors to the faces of $T_{i,j}$ and let  $\vec{n}_1^A,\vec{n}_2^A,\vec{n}_3^A,\vec{n}_4^A$ denote the normal vectors of $T_{i,j}^A$. Without loss of generality, suppose that $\vec{n}_1,\vec{n}_2,\vec{n}_3,\vec{n}_4$ all point inwards $Q_{i,j}$, numbered with the counterclockwise orientation. Then the antipodal, $Q_{i,j}^A$ has the same normal vectors but point outwards numbered with the clockwise orientation. The left and the right faces of $Q_{i,j}$ have normal vectors $\vec{n}_1$ and $\vec{n}_3$. Since $Q_{i,j}^A$ is a reflection of $Q_{i,j}$ through the center of the sphere, the normal vectors to $Q_{i,j}^A$ must be related to the normal vectors of $Q_{i,j}$ as follows: $\vec{n}_1^A=-\vec{n}_3, \vec{n}_2^A=-\vec{n}_2,\vec{n}_3^A,=-\vec{n}_1, \vec{n}_4^A=-\vec{n}_4$.

We will examine if the normal vectors defined by $T_{i,j}^A$ satisfy these relations.
Notice that by definition

\begin{equation}
\vec{n}_1^A=\frac{\vec{r}_{i,j^A}\times \vec{r}_{i,(j+1)^A}}{||\vec{r}_{ij^A}\times \vec{r}_{i,(j+1)^A}||}\nonumber
\end{equation}

\noindent where $\vec{r}_{ij^A}=\vec{p}_i-\vec{p}_j^A=\vec{p}_i-\vec{p}_{i+1}+(\vec{p}_{j}-\vec{p}_i)=\vec{p}_j-\vec{p}_{i+1}=-\vec{r}_{i+1,j}$, $\vec{r}_{i,(j+1)^A}=\vec{p}_i-\vec{p}_{j+1}^A=\vec{p}_i-\vec{p}_{i+1}+(\vec{p}_{j+1}-\vec{p}_i)=\vec{p}_{j+1}-\vec{p}_i=-\vec{r}_{i,j+1}$.
Thus, $\vec{n}_1^A=-\vec{n}_3$.

Similarly, one can verify that the normal vectors of $T_{i,j}^A$, $\vec{n}_1^A,\vec{n}_2^A,\vec{n}_3^A,\vec{n}_4^A$ satisfy:  $\vec{n}_1^A=-\vec{n}_3$, $\vec{n}_2^A=-\vec{n}_2$, $\vec{n}_3^A=-\vec{n}_1$ and $\vec{n}_4^A=-\vec{n}_4$.

\end{proof}

The following theorem determines the probability that three edges, two of which are consecutive, cross in a projection direction.

\begin{figure}
   \begin{center}
\includegraphics[width=0.7\textwidth]{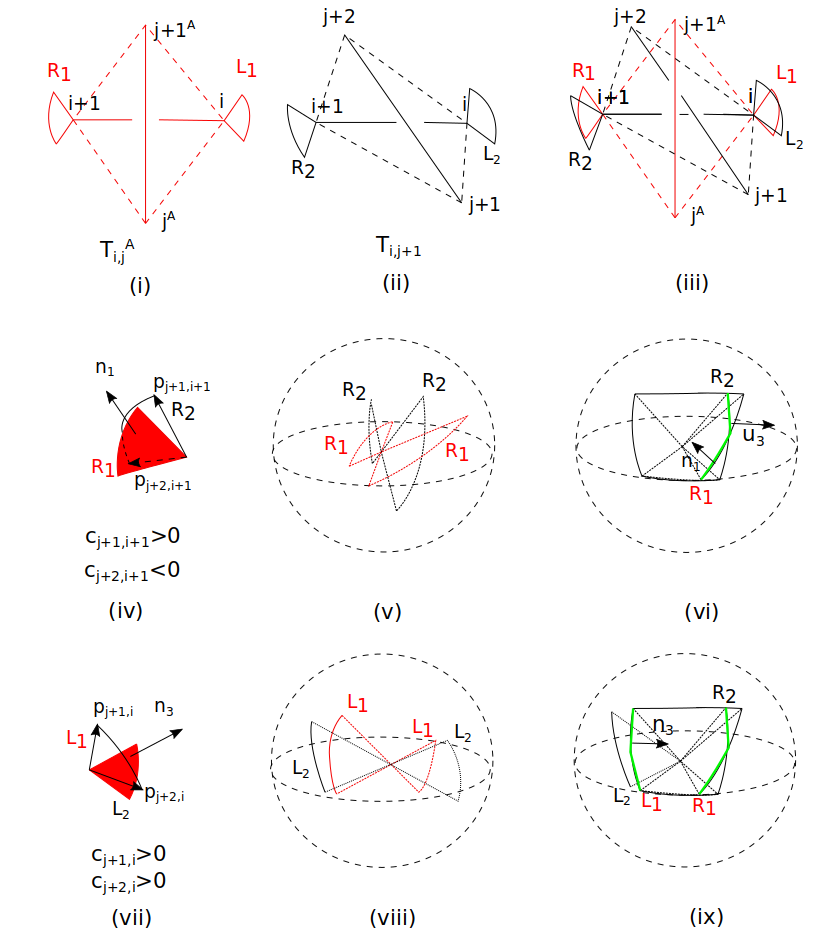}
     \caption{The quadrangle $Q_{i,j,j+1}=(Q_{i,j}\cap Q_{i,j+1})\cup(Q_{i,j}^A\cap Q_{i,j+1})$ contains the vectors that define projections of $e_i,e_j$ and $e_i,e_{j+1}$ both intersect. This Figure shows the procedure for determining $Q_{i,j}^A\cap Q_{i,j+1}$ in the case where $\epsilon_{i,j}=\epsilon_{i,j+1}$. $Q_{i,j}^A\cap Q_{i,j+1}$ is bounded by the great circles defined by the intersection of the faces of the quadrilaterals $T_{i,j}^A$ and $T_{i,j+1}$.(i) The quadrilateral $T_{i,j}^A$ (ii) The quadrilateral $T_{i,j+1}$ (iii) The relative positions of $T_{i,j}^A$ and $T_{i,j+1}$. (iv-ix) At the vertices $i,i+1$, we can define the left and right spherical faces of $Q_{i,j}^A$ and $Q_{i,j+1}$. To find the left and right faces of , we examine the intersection of the spherical faces at $i$ and at $i+1$ (see Definition \ref{defn1}). Let $\vec{p}_{j+1,i+1},\vec{p}_{j+2,i+1}$ be the vectors that connect vertex $j+1$ and vertex $j+2$ to $i+1$. In this example, $c_{j+1,i+1}=(\vec{p}_{j+1,i+1}\cdot\vec{n}_1)\epsilon_{i,j}>0$ and $c_{j+2,i+1}=(\vec{p}_{j+2,i+1}\cdot\vec{n}_1)\epsilon_{i,j}<0$ and the spherical faces $R_1$, $R_2$ intersect and they both bound $Q_1$. Similarly, in this example, $c_{j+1,i}=\vec{p}_{j+1,i}\cdot\vec{n}_3>0$ and $c_{j+2,i}=\vec{p}_{j+2,i}\cdot\vec{n}_3>0$ and only the spherical face $L_1$ bounds $Q_{i,j}^A\cap Q_{i,j+1}$ . }
     \label{fig:qk21}
   \end{center}
\end{figure}

\begin{theorem}\label{qijj+1}
Let $e_i,e_j,e_{j+1}$ denote three edges in 3-space. Then the joint probability of crossing between the projections of $e_i,e_j$ and $e_i,e_{j+1}$, is equal to $\frac{1}{2\pi}A(Q_{i,j,j+1})$, where $Q_{i,j,j+1}$ is given in Table \ref{tableqij1} and  Table \ref{tableqij2}.
\end{theorem}

\begin{proof}

Let $T_{i,j}$ and $T_{i,j+1}$ be the two quadrilaterals formed by $e_i,e_j$ and $e_i,e_{j+1}$, where $e_i$ connects vertex $i$ to $i+1$, $e_j$ connect vertex $j$ to $j+1$ and vertex $j+1$ to $j+2$. Let $\vec{n}_1,\vec{n}_2,\vec{n}_3,\vec{n}_4$ denote the normal vectors to the faces of $T_{i,j}$ and $\vec{u}_1,\vec{u}_2,\vec{u}_3,\vec{u}_4$ denote the normal vectors to the faces of $T_{i,j+1}$. 
The normal vectors defined by the quadrilaterals define great circles which intersect to form the corresponding quadrangles. Each pair of great circles intersects at 2 antipodal points on the sphere, but due to the connectivity of the edges, there are also points where more than two great circles cross. These great circles correspond to faces of the tetrahedrals that share a common edge. The normal vectors to more than two great circles that intersect and their common edge on the tetrahedrals are shown in Table \ref{table0}. Due to the connectivity of the edges $e_j,e_{j+1}$, $T_{i,j}$ and $T_{i,j+1}$ share a common face, the one formed by the vertices $i,i+1,j+1$, which implies that the normal vectors $\vec{n}_2$ and $\vec{u}_4$ are collinear.  Thus, the great circles $Q_{i,j}$ and $Q_{i,j+1}$ or the great circles $Q_{i,j}^A$ and $Q_{i,j+1}$ share a common face, which implies that either $Q_{i,j}\cap Q_{i,j+1}=\emptyset$ or $Q_{i,j}^A\cap Q_{i,j+1}=\emptyset$, depending on the sign of $\epsilon_{i,j}$ and $\epsilon_{i,j+1}$.

\begin{table}[H]
\centering
\begin{tabular}{|l|l|}
\hline
\textbf{great circles}           & \textbf{common edge}\\
\hline
$\vec{n}_2,\vec{n}_4,\vec{u}_2$                   & $i,i+1$\\
$\vec{n}_1,\vec{n}_2,\vec{u}_1$                   & $i,j+1$\\
$\vec{n}_2,\vec{n}_3,\vec{u}_3$                   & $i+1,j+1$\\
\hline
\end{tabular}
\caption{Vectors perpendicular to great circles that contain a common edge.}
\label{table0}
\end{table}

Suppose $\epsilon_{i,j}=\epsilon_{i,j+1}$ (see Figure \ref{fig:eij} for an illustrative example). Then in order for the projections of $e_i,e_j,e_{j+1}$ to intersect, $e_i$ must pierce the triangle defined by $e_j,e_{j+1}$. To check this we examine the signs of $w_0=(\vec{v}_3\times(-\vec{n}_1))\cdot(\vec{v}_3\times\vec{n}_3)$ and $w=(\vec{u}_2\times (-\vec{n}_2))\cdot(\vec{u}_2\times \vec{n}_4)$, where $\vec{v}_3=\vec{p}_{i,j+2}\times\vec{p}_{j+1,j+2}$. If $w_0>0$, then $Q_{i,j,j+1}=\emptyset$. The faces with normal vectors $\vec{u}_2,\vec{n}_2,\vec{n}_4$ share a common edge and, if $(\vec{u}_2\times (-\vec{n}_2))\cdot(\vec{u}_2\times \vec{n}_4)>0$, then both $\vec{n}_2$ and $\vec{n}_4$  do not intersect $T_{i,j+1}$, so $A(Q_{i,j,j+1})=0$ (see Figure \ref{fig:eij}).  Suppose that $w_0<0$ and $w<0$. In that case $\vec{n}_2=-\vec{u}_4$ and the face $i,i+1,j$ contains the only points in the intersection of $T_{i,j}$ with $T_{i,j+1}$, thus  $A(Q_{i,j}\cap Q_{i,j+1})=0$. We therefore examine the intersection of $T_{i,j}^A\cap T_{i,j+1}$, which determines $Q_{i,j}^A\cap Q_{i,j+1}$ (see Theorem \ref{antipodal}). Since $\vec{n}_2^A=-\vec{n}_2=\vec{u}_4$, and $T_{i,j}^A$ is the antipodal of $T_{i,j}$, the face of $T_{i,j}^A$ with normal vector $\vec{n}_2^A$ and the face of $T_{i,j+1}$ with normal vector $\vec{u}_4$ lie in the same plane but do not intersect (as shown in Figure \ref{fig:eij}). 
Since $w<0$ we know that $A(Q_{i,j,j+1})\neq0$ and it is formed by $\vec{u}_2,\vec{n}_4$ and, at least some of, the vectors $\vec{u}_1,\vec{u}_3,\vec{n}_1,\vec{n}_3$.

To find the other faces of $Q_{i,j,j+1}$, we think at the level of right and left spherical faces of the tetrahedra $T_{i,j}^A$ and $T_{i,j+1}$. These faces share a common vertex, the vertex $i$ and $i+1$, respectively. The spherical face of $T_{i,j}^A$ (resp. $T_{i,j+1}$) at $i$ has normal vector $\vec{n}_1^A$ (resp. $\vec{u}_1$) and the spherical face of $T_{i,j}^A$ (resp. $T_{i,j+1}$) at $i+1$ has normal vector $\vec{n}_3^A$ (resp. $\vec{u}_3$).  We compare the direction of the edges $\vec{p}_{i,j+1},\vec{p}_{i,j+2}$ at the vertex $i$ with the direction of $\vec{n}_1^A$ to determine the position of the spherical face that they define (the one with normal vector $\vec{u}_1$) relative to the one with normal vector $\vec{n}_1^A$ Taking into account that $\vec{n}_1^A=-\vec{n}_3$ and $\vec{n}_3^A=-\vec{n}_1$, and whether these vectors point inwards or outwards $Q_{i,j,j+1}$, depending on the sign of $\epsilon_{ij}$, we let $c_{j+1,i+1}=(\vec{p}_{j+1,i+1}\cdot\vec{n}_1)\epsilon_{ij}$, $c_{j+2,i+1}=(\vec{p}_{j+2,i+1}\cdot\vec{n}_1)\epsilon_{ij}$, $c_{j+1,i}=(\vec{p}_{j+1,i}\cdot\vec{n}_3)\epsilon_{ij}$, $c_{j+2,i}=(\vec{p}_{j+2,i}\cdot\vec{n}_3)\epsilon_{ij}$ (see Figure \ref{fig:qk21} for an illustrative example) and we think as follows: If $c_{j+1,i+1}\cdot c_{j+2,i+1}>0$, then only one of the great circles with normal vectors $\vec{n}_1,\vec{u}_3$ will be on the boundary of $Q_{i,j,j+1}$ and if $c_{j+1,i+1}\cdot c_{j+2,i+1}<0$, the spherical faces intersect and both bound $Q_{i,j,j+1}$. 
Namely, if  $c_{j+1,i+1}>0$ and $c_{j+2,i+1}>0$, only $\vec{n}_1$ and not $\vec{u}_3$ are in the boundary of $Q_{i,j,j+1}$,  if $c_{j+1,i+1}<0$ and $c_{j+2,i+1}<0$, then only $\vec{u}_3$ and not $\vec{n}_1$ is the boundary of $Q_{i,j,j+1}$.  If $c_{j+1,i+1}>0$ and $c_{j+2,i+1}<0$ then both $\vec{n}_1,\vec{u}_3$ are in the boundary of $Q_{i,j,k}$ in the following counterclockwise order $\vec{n}_4,\vec{n}_1,-\vec{u}_3,-\vec{u}_2$. If $c_{j+1,i+1}<0$ and $c_{j+2,i+1}>0$ then both $\vec{u}_3,\vec{n}_1$ are in the boundary of $Q_{i,j,j+1}$ in the following counterclockwise order $\vec{n}_4,-\vec{u}_3,\vec{n}_1,-\vec{u}_2$.

In a similar way we find which of the spherical edges formed by $T_{i,j+1},T_{i,j}^A$ at the vertex $i$ form the other side of the boundary of $Q_{i,j,j+1}$. We notice that in this case, if $c_{j+1,i}>0,c_{j+2,i}>0$, then only $\vec{n}_3$ is in the boundary of $Q_{i,j+1}$, if $c_{j+1,i}>0,c_{j+2,i}<0 $, $\vec{n}_3,\vec{u}_1$ both are in the following order counterclockwise $-\vec{u}_2,-\vec{u}_1,\vec{n}_3,\vec{n}_4$. If $c_{j+1,i}<0,c_{j+2,i}>0$, they both are but in the following order $-\vec{u}_2,\vec{n}_3,-\vec{u}_1,\vec{n}_4$. If $c_{j+1,i}<0,c_{j+2,i}<0$, only $\vec{u}_1$ is in the boundary of $Q_{i,j,j+1}$ (see Figure \ref{fig:qk21}).

\begin{figure}
   \begin{center}
\includegraphics[width=0.7\textwidth]{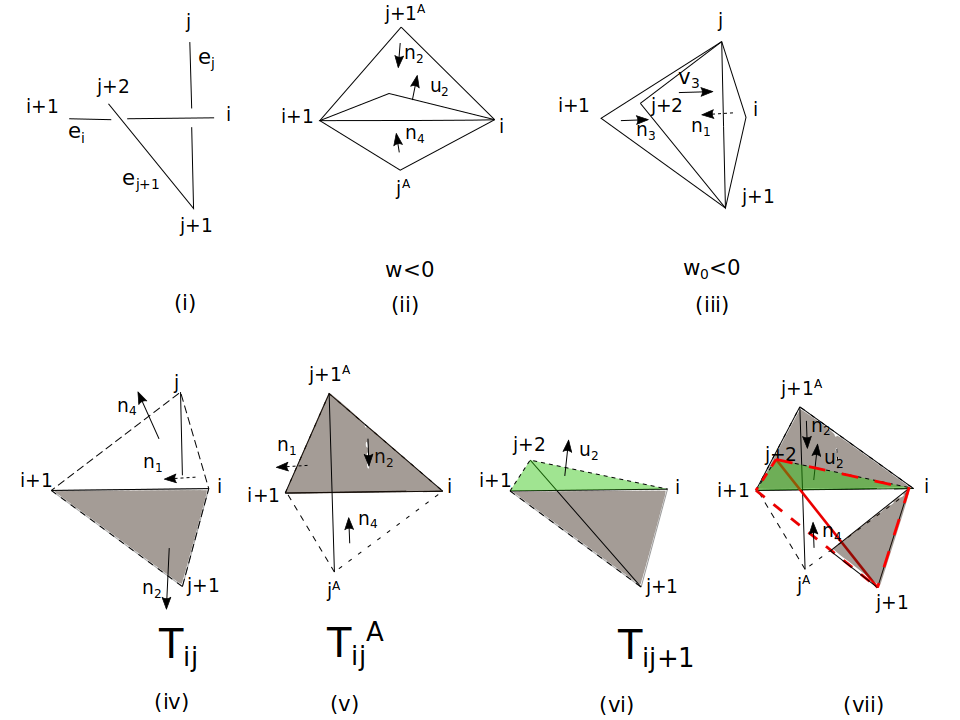}
     \caption{(i-iii) A configuration where $\epsilon_{i,j}=\epsilon_{i,j+1}$, $w<0$ and $w_0<0$. (iv) The tetrahedral formed by $e_i,e_j$, we denote $T_{i,j}$. (v) The antipodal of $T_{i,j}$, we denote $T_{i,j}^A$ (see Definition \ref{defn1} and Theorem \ref{antipodal}). (vi) The tetrahedral formed by $e_i,e_{j+1}$. In this case, $T_{i,j}\cap T_{i,j+1}=\emptyset$. If $w=(\vec{u}_2\times (-\vec{n}_2))\cdot(\vec{u}_2\times \vec{n}_4)>0$, then $T_{i,j}^A\cap T_{i,j+1}=\emptyset$ as well, giving $A(Q_{i,j,j+1})=0$. (vii) If $w<0$, then $T_{i,j}^A\cap T_{i,j+1}\neq\emptyset$, giving $A(Q_{i,j,j+1})\neq0$ (see proof of Theorem \ref{qijj+1}).}
     \label{fig:eij}
   \end{center}
\end{figure}

\begin{figure}
   \begin{center}
\includegraphics[width=0.8\textwidth]{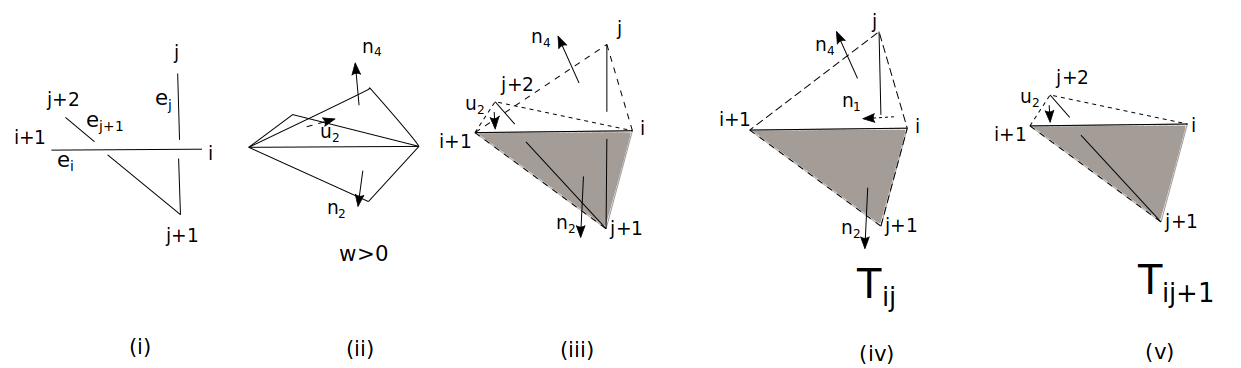}
     \caption{(i) A configuration where $\epsilon_{i,j}=-\epsilon_{i,j+1}$. (ii) In this configuration $w>0$. (iii) The tetrahedral formed by $e_i,e_j$, $T_{i,j}$ (iv) The tetrahedral formed by $e_i,e_{j+1}$, $T_{i,j+1}$. (v) In this case, $T_{i,j}^A\cap T_{i,j+1}=\emptyset$ and $T_{i,j}\cap T_{i,j+1}\neq\emptyset$. If $w=(\vec{u}_2\times (-\vec{n}_2))\cdot(\vec{u}_2\times \vec{n}_4)<0$, then two faces of $Q_{i,j,j+1}$ have normal vectors $\vec{n}_2,\vec{u}_2$, otherwise, it is $\vec{n}_2,\vec{n}_4$ (see proof of Theorem \ref{qijj+1}).}
     \label{fig:eij2}
   \end{center}
\end{figure}

Suppose that $\epsilon_{i,j}=-\epsilon_{i,j+1}$, then $A(Q_{i,j}^A\cap Q_{i,j+1})=0$ and $A(Q_{i,j,j+1})=A(Q_{i,j}\cap Q_{i,j+1})\neq0$ for all values of $w$ and one face of $Q_{i,j,j+1}$ has normal vector $n_2$ (see Figure \ref{fig:eij2} for an illustrative example). 
If $w<0$ the other face is $\vec{u}_2$ and if $w\geq0$,  it is $\vec{n}_4$.  The left and right faces are of both quadrilaterals share a common edge (the extension of the edges $j+1,i$ and $j+1,i+1$) and thus do not intersect. So, only one of each will be the boundary of $Q_{i,j,j+1}$. To determine which, we check if $c_{j+2,i}=(\vec{p}_{j+2,i}\cdot\vec{n}_1)\epsilon_{ij}>0$, then $\vec{u}_1$ is the boundary, otherwise it is $\vec{n}_1$. If $c_{j+2,i+1}=(\vec{p}_{j+2,i}\cdot\vec{n}_3)\epsilon_{ij}>0$, then $\vec{u}_3$ is the boundary, otherwise it is $\vec{n}_3$.

\end{proof}

\begin{table}
\centering
\begin{tabular}{|l|l|}
\hline
$\epsilon_{i,j}=\epsilon_{i,j+1}, w<0, w_0<0$ & $Q_{i,j,j+1}$\\
\hline
 $c_{j+1,i+1}>0, c_{j+2,i+1}>0, c_{j+1,i}>0, c_{j+2,i}>0$                   & $(n_4,\vec{n}_1,-\vec{u}_2,\vec{n}_3)$\\
 $c_{j+1,i+1}>0, c_{j+2,i+1}>0, c_{j+1,i}<0, c_{j+2,i}<0$                   & $(\vec{n}_4,\vec{n}_1,-\vec{u}_2,-\vec{u}_1)$\\
$c_{j+1,i+1}>0, c_{j+2,i+1}>0, c_{j+1,i}>0, c_{j+2,i}<0$                   & $(\vec{n}_4,\vec{n}_1,-\vec{u}_2,-\vec{u}_1,\vec{n}_3)$\\
 $c_{j+1,i+1}>0, c_{j+2,i+1}>0, c_{j+1,i}<0, c_{j+2,i}>0$                   & $(\vec{n}_4,\vec{n}_1,-\vec{u}_2,\vec{n}_3,-\vec{u}_1)$\\
 $c_{j+1,i+1}<0, c_{j+2,i+1}<0, c_{j+1,i}>0, c_{j+2,i}>0$                   & $(\vec{n}_4,-\vec{u}_3,-\vec{u}_2,\vec{n}_3)$\\
$c_{j+1,i+1}<0, c_{j+2,i+1}<0, c_{j+1,i}<0, c_{j+2,i}<0$                   & $(\vec{n}_4,-\vec{u}_3,-\vec{u}_2,-\vec{u}_1)$\\
$c_{j+1,i+1}<0, c_{j+2,i+1}<0, c_{j+1,i}>0, c_{j+2,i}<0$                   & $(\vec{n}_4,-\vec{u}_3,-\vec{u}_2,-\vec{u}_1,\vec{n}_3)$\\
$c_{j+1,i+1}<0, c_{j+2,i+1}<0, c_{j+1,i}<0, c_{j+2,i}>0$                   & $(\vec{n}_4,-\vec{u}_3,-\vec{u}_2,\vec{n}_3,-\vec{u}_1)$\\
$c_{j+1,i+1}>0, c_{j+2,i+1}<0, c_{j+1,i}>0, c_{j+2,i}>0$                   & $(\vec{n}_4,\vec{n}_1,-\vec{u}_3,-\vec{u}_2,\vec{n}_3)$\\
$c_{j+1,i+1}>0, c_{j+2,i+1}<0, c_{j+1,i}<0, c_{j+2,i}<0$                   & $(\vec{n}_4,\vec{n}_1,-\vec{u}_3,-\vec{u}_2,-\vec{u}_1)$\\
$c_{j+1,i+1}>0, c_{j+2,i+1}<0, c_{j+1,i}>0, c_{j+2,i}<0$                   & $(\vec{n}_4,\vec{n}_1,-\vec{u}_3,-\vec{u}_2,-\vec{u}_1,\vec{n}_3)$\\
$c_{j+1,i+1}>0, c_{j+2,i+1}<0, c_{j+1,i}<0, c_{j+2,i}>0$                   & $(\vec{n}_4,\vec{n}_1,-\vec{u}_3,-\vec{u}_2,\vec{n}_3,-\vec{u}_1)$\\
$c_{j+1,i+1}<0, c_{j+2,i+1}>0, c_{j+1,i}>0, c_{j+2,i}>0$                   & $(\vec{n}_4,-\vec{u}_3,\vec{n}_1,-\vec{u}_2,\vec{n}_3)$\\
$c_{j+1,i+1}<0, c_{j+2,i+1}>0, c_{j+1,i}<0, c_{j+2,i}<0$                   & $(\vec{n}_4,-\vec{u}_3,\vec{n}_1,-\vec{u}_2,-\vec{u}_1)$\\
$c_{j+1,i+1}<0, c_{j+2,i+1}>0, c_{j+1,i}>0, c_{j+2,i}<0$                   & $(\vec{n}_4,-\vec{u}_3,\vec{n}_1,-\vec{u}_2,-\vec{u}_1,\vec{n}_3)$\\
$c_{j+1,i+1}<0, c_{j+2,i+1}>0, c_{j+1,i}<0, c_{j+2,i}>0$                   & $(\vec{n}_4-\vec{u}_3,\vec{n}_1,-\vec{u}_2,\vec{n}_3,-\vec{u}_1)$\\
\hline
$\epsilon_{i,j}=\epsilon_{i,j+1}, w>0$ or $w_0>0$         & $Q_{i,j,j+1}$\\
\hline
 & $\emptyset$\\
\hline
\end{tabular}
\caption{The spherical polygon $Q_{i,j,j+1}$ in the case where the signs satisfy $\epsilon_{i,j}=\epsilon_{i,j+1}$, depending on the conformation. The spherical polygon $Q_{i,j,j+1}$  contains the vectors which define planes where the projections of $e_i,e_j$ and $e_i,e_{j+1}$ both cross.  $(\vec{w}_1,\vec{w}_2,\dotsc,\vec{w}_n)$ denotes the spherical polygon bounded by the great circles with normal vectors $\vec{w}_i$,  $i=1,\dotsc, n$, in the counterclockwise orientation, (see Definition \ref{defn1} and proof of Theorem \ref{qijj+1}).}
\label{tableqij1}
\end{table}

\begin{table}[H]
\centering
\begin{tabular}{|l|l|}
\hline
$\epsilon_{i,j}=-\epsilon_{i,j+1}, w<0$   & $Q_{i,j,j+1}$\\
\hline
 $c_{j+2,i}>0,c_{j+2,i+1}>0$                                & $(\vec{n}_2,-\vec{u}_1,-\vec{u}_2,-\vec{u}_3)$\\
 $c_{j+2,i}<0,c_{j+2,i+1}<0$                                & $(\vec{n}_2,\vec{n}_1,-\vec{u}_2,\vec{n}_3)$\\
 $c_{j+2,i}<0,c_{j+2,i+1}>0$                                & $(\vec{n}_2,\vec{n}_1,-\vec{u}_2,-\vec{u}_3)$\\
 $c_{j+2,i}>0,c_{j+2,i+1}<0$                                & $(\vec{n}_2,-\vec{u}_1,-\vec{u}_2,\vec{n}_3)$\\
\hline
$\epsilon_{i,j}=-\epsilon_{i,j+1}, w>0$   & $Q_{i,j,j+1}$\\
\hline
 $c_{j+2,i}>0,c_{j+2,i+1}>0$                                & $(\vec{n}_2,-\vec{u}_1,\vec{n}_4,-\vec{u}_3)$\\
 $c_{j+2,i}<0,c_{j+2,i+1}<0$                                & $(\vec{n}_2,\vec{n}_1,\vec{n}_4,\vec{n}_3)$\\
$c_{j+2,i}<0,c_{j+2,i+1}>0$                                & $(\vec{n}_2,\vec{n}_1,\vec{n}_4,-\vec{u}_3)$\\
 $c_{j+2,i}>0,c_{j+2,i+1}<0$                                & $(\vec{n}_2,-\vec{u}_1,\vec{n}_4,\vec{n}_3)$\\
\hline
\end{tabular}
\caption{The spherical polygon $Q_{i,j,j+1}$ in the case where the signs satisfy $\epsilon_{i,j}=-\epsilon_{i,j+1}$, depending on the conformation. The spherical polygon $Q_{i,j,j+1}$  contains the vectors which define planes where the projections of $e_i,e_j$ and $e_i,e_{j+1}$ both cross.  $(\vec{w}_1,\vec{w}_2,\dotsc,\vec{w}_n)$ denotes the spherical polygon bounded by the great circles with normal vectors $\vec{w}_i$,  $i=1,\dotsc, n$, in the counterclockwise orientation, (see Definition \ref{defn1} and proof of Theorem \ref{qijj+1}).}
\label{tableqij2}
\end{table}

\begin{theorem}\label{knotoid4}
Let $E_4$ denote a polygonal chain of 4 edges in 3 space. The probability that its projection on a random projection direction is the non-trivial knotoid k2.1 is equal to 

\begin{equation}
P(K((E_4)_{\xi})=k2.1)=P((E_4)_{\xi}=k2.1_{Bi})+P((E_4)_{\xi}=k2.1_{Bii})
\end{equation}

where $k2.1_{Bi}$ and $k2.1_{Bii}$ are the two possible k2.1 diagrams (see case B(i) and case B(ii) in Figure \ref{fig:4knotoids}) and where

\begin{align}\label{pk21}
P((E_4)_{\xi}=k2.1_{Bi})=
\begin{cases}
0, \text{ if } \epsilon_{1,3}\neq \epsilon_{1,4}\text{ or } w>0 \text{ or } w_0>0\\
\frac{1}{2\pi} Area(\vec{v}_3,-\vec{v}_2,-\vec{u}_2) \text{ if } c_{4,1}<0,w<0,w_0<0\\
\frac{1}{2\pi} Area(\vec{v}_3,-\vec{v}_2,\vec{n}_1,-\vec{u}_2) \text{ if } c_{4,1}>0,w<0,w_0<0 
\end{cases}
\end{align}

\noindent where $c_{4,1}=(\vec{p}_{4,1}\cdot\vec{n}_1)\epsilon_{1,3}$,  $w=(u_2\times (-n_2))\cdot (u_2\times n_4)$, $w_0=(\vec{v}_3\times(-\vec{n}_1))\cdot(\vec{v}_3\times\vec{n}_3)$ and the vectors $\vec{u}_2,\vec{n}_2,\vec{n}_4,\vec{v}_3,\vec{v}_2$ and $\vec{n}_1$ are normal to the planes containing the vertices $014,013,021,243,241$, and $023$, respectively. $P((E_4)_{\xi}=k2.1_{Bii})=P((R(E_4))_{\xi}=k2.1_{Bi})$, where $R(E_4)$
 is the walk $E_4$ with reversed orientation.
\end{theorem}

\begin{proof}

Using the notation of Proposition \ref{number4},  the probability of having a non-trivial knotoid in a random projection is equal to the probability of case B (i) or (ii) shown in Figure \ref{fig:4knotoids} with crossings $\epsilon_{1,3}=\epsilon_{1,4}$. By Proposition \ref{case} if one of the two is non-zero the other is zero. Thus, it suffices to find the probability that a projection of a polygonal chain of 4 edges is of the form  case B (i) with $\epsilon_{1,3}=\epsilon_{1,4}$, and if that probability is equal to 0, then one needs to compute the probability that it is of the type case B (ii) with $\epsilon_{1,4}=\epsilon_{2,4}$. To find a closed formula for these cases, it suffices to find a closed formula for the probability that it is nontrivial case B (i), since the same formula applied to the polygonal chain with reversed orientation of edges, will give the probability of getting case B (ii), $\epsilon_{1,4}=\epsilon_{2,4}$.

Let $\epsilon_{i,j,\vec{\xi}}$ denote the sign of the crossing between the projections of the edges $e_i,e_j$ to the plane with normal vector $\vec{\xi}$. This variable takes the values $\epsilon_{i,j,\vec{\xi}}=\epsilon_{i,j}$ when the projections of $e_i,e_j$ cross in the plane with normal vector $\vec{\xi}$ and $\epsilon_{i,j,\vec{\xi}}=0$ when the projections of $e_i,e_j$ do not cross in that plane. 

The condition for $(E_4)_{\vec{\xi}}$ being case B (i) with $\epsilon_{1,3}=\epsilon_{1,4}$ is: $\epsilon_{1,3,\vec{\xi}}=\epsilon_{1,4,\vec{\xi}}\neq0$, $\epsilon_{2,4,\vec{\xi}}=0$ and $(e_4)_{\vec{\xi}}$ lies in the side of $(e_3)_{\vec{\xi}}$ that is inside the k2.1 bounded region. Without loss of generality, let us focus in the case of $\epsilon_{1,3,\vec{\xi}}=\epsilon_{1,4,\vec{\xi}}=1$.
Let us denote these conditions as: $(\epsilon_{1,3,\vec{\xi}}=1)\cap (\epsilon_{1,4,\vec{\xi}}=1)\cap (\epsilon_{2,4,\vec{\xi}})=0\cap C_{e_4,e_3}$, where $C_{e_4,e_3}$ denotes the condition on $(e_4)_{\vec{\xi}}$ being in the side of $(e_3)_{\vec{\xi}}$ that is inside the region bounded by the projection of the edges $e_1,e_2,e_3$. Thus:

\begin{align}
P(K((E_4)_{\vec{\xi}})=k2.1)&=P(\epsilon_{1,3,\vec{\xi}}=1\cap \epsilon_{1,4,\vec{\xi}}=1\cap \epsilon_{2,4,\vec{\xi}}=0\cap C_{e_4,e_3})\nonumber\\
&=\frac{A(Q_{1,3,4}\cap((S^2\setminus Q_{2,4})\cap C_{e_4,e_3})}{2\pi}
\end{align}

\noindent where we canceled a factor 2 in the numerator which arises because antipodal vectors give the same diagram.

\begin{figure}
   \begin{center}
\includegraphics[width=0.7\textwidth]{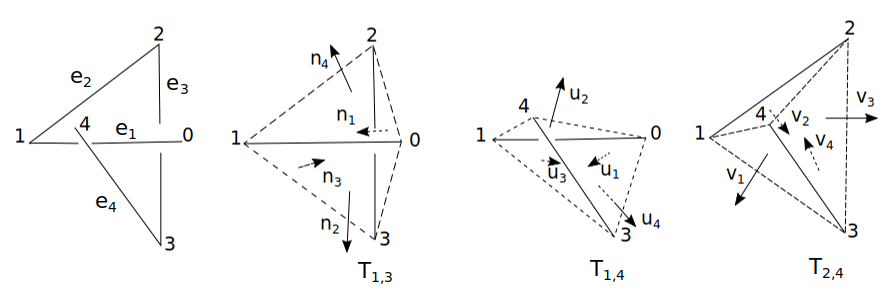}
     \caption{A chain with four edges, $e_1,e_2,e_3,e_4$ that connect the vertices $0,1,2,3,4$ in the case where $\epsilon_{1,3}=\epsilon_{1,4}=-1$. The three pairs of edges, $e_1,e_3$, $e_1,e_4$ and $e_2,e_4$, define 3 quadrilaterals $T_{1,3},T_{1,4}$ and $T_{2,4}$, respectively. These quadrilaterals define three quadrangles (and their antipodals) on the unit sphere, $Q_{13},Q_{14}$ and $Q_{24}$, respectively, which contain vectors which define projections where the projection of the corresponding pairs cross (see Figure \ref{fig:sphere}).}
     \label{fig:quad}
   \end{center}
\end{figure}

\begin{figure}
   \begin{center}
     \includegraphics[width=0.9\textwidth]{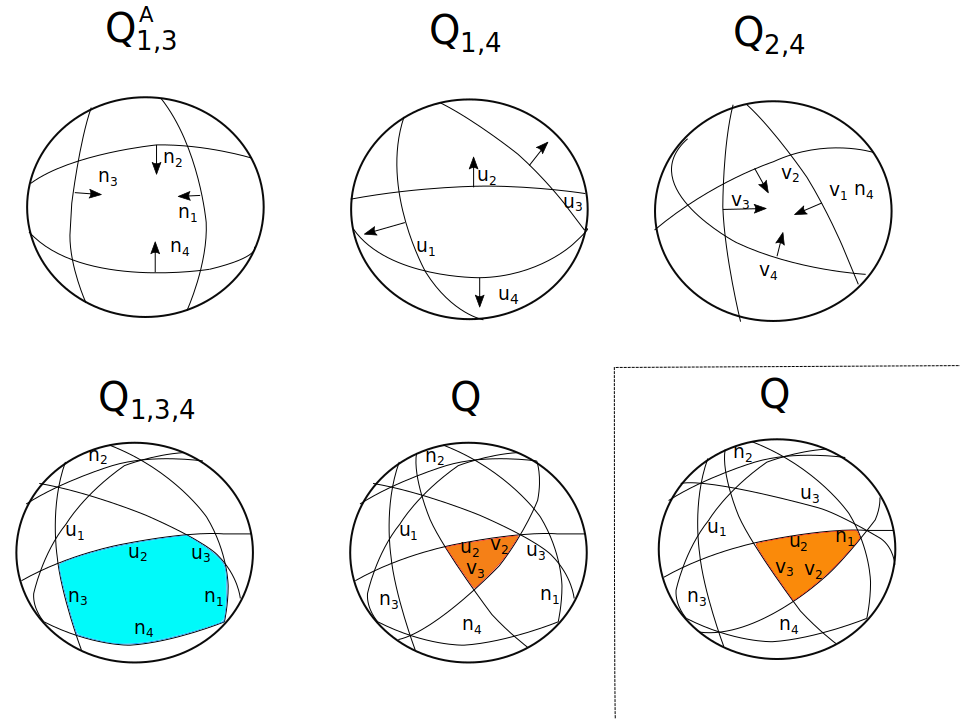}
     \caption{The spherical quadrangles $Q_{1,3}^A$, $Q_{1,4}$ and $Q_{2,4}$ defined by the faces of the quadrilaterals shown in Figure \ref{fig:quad}. $Q_{1,3,4}=Q_{1,3}\cap Q_{1,4}$ contains the vectors which define projections of $e_1,e_3,e_4$ where both pairs $e_1,e_3$ and $e_1,e_4$ cross. $Q$ is those vectors which define projections where the projection of the chain gives the knotoid $k2.1$ (configuration case B(i) from Figure \ref{fig:4knotoids}). Depending on the positions of the great circles, the resulting $Q$ could be that shown in the margin. }
     \label{fig:sphere}
   \end{center}
\end{figure}

Since $\epsilon_{1,3}=\epsilon_{1,4}$,  $Q_{1,3,4}=Q_{1,3}^A\cap Q_{1,4}$ (see Theorem \ref{qijj+1} for $i=1,j=3$).
The quadrangles $Q_{1,3}^A$, $Q_{1,4}$ and $Q_{2,4}$ are generated by the quadrilaterals $T_{1,3}^A$, $T_{1,4}$ and $T_{2,4}$ (see Figures \ref{fig:quad} and \ref{fig:sphere}). The normal vectors of $T_{1,3}^A$ are $\vec{n}_1^A=-\vec{n}_3,\vec{n}_2^A=-\vec{n}_2,\vec{n}_3^A=-\vec{n}_1,\vec{n}_4^A=-\vec{n}_4$, which define the faces of the quadrilateral in a counterclockwise order, all pointing outward $Q_{1,3}^A$. The normal vectors of $T_{1,4}$ are $\vec{u}_1,\vec{u}_2,\vec{u}_3,\vec{u}_4$, where $\vec{n}_2=-\vec{u}_4$, and all point outward $Q_{1,4}$ and the normal vectors of $T_{2,4}$ are $\vec{v}_1,\vec{v}_2,\vec{v}_3,\vec{v}_4$ in counterclockwise order, where $\vec{n}_3=\vec{v}_4,\vec{u}_3=-\vec{v}_1$, and they all point inward $Q_{2,4}$ (see Figure \ref{fig:sphere}). So, in total we have 9 vectors, which define 9 great circles on $S^2$. The normal vectors to more than two great circles that intersect and the vertices that define their common edge are shown in Table \ref{table2}.

\begin{table}[H]
\centering
\begin{tabular}{|l|l|}
\hline
\textbf{great circles}           & \textbf{common quadrilateral edge}\\
\hline
$\vec{n}_2,\vec{n}_4,\vec{u}_2$                   & 01\\
$\vec{n}_1,\vec{n}_2,\vec{u}_1$                   & 03\\
$\vec{n}_3,\vec{n}_4,\vec{v}_2$                   & 12\\
$\vec{n}_2,\vec{n}_3,\vec{u}_3,\vec{v}_1$               & 13\\
$\vec{u}_2,\vec{u}_3,\vec{v}_2$                   & 14\\
$\vec{n}_1,\vec{n}_3,\vec{v}_3$                   & 24\\
$\vec{u}_1,\vec{u}_3,\vec{v}_3$                   & 34\\
\hline
\end{tabular}
\caption{Vectors perpendicular to great circles that contain a common vector.}
\label{table2}
\end{table}

$Q_{1,3,4}$ was computed in Theorem \ref{qijj+1}. $Q_{1,3,4}$ is formed by $\vec{n}_4,\vec{u}_2$ and some of $\vec{n}_3,\vec{u}_3,\vec{n}_1,\vec{u}_1$.
To find $Q_{1,3,4}\cap((S^2\setminus Q_{2,4})\cap C_{e_4,e_3})$, we think as follows: First we notice that if $Q_{1,3,4}\neq\emptyset$, then $Q_{2,4}\subset Q_{1,4}$. This quadrangle will include great circles defined by the normal vectors $\vec{v}_{i}$ (that involve the edges $e_4,e_2$).  
Note that the great circles with normal vectors  $\vec{v}_2$, $\vec{n}_4$ and $\vec{n}_3$ intersect (bottom left corner of $Q_{134}$, see Figure \ref{fig:sphere}) and the great circles with normal vectors $\vec{v}_2$, $\vec{u}_2$ and $\vec{u}_3$ also intersect (top right corner of $Q_{134}$). Thus, $\vec{v}_2$ intersects the interior of $Q_{1,3,4}$. Since $\vec{v}_2$ bounds $Q_{2,4}$ and points inwards $Q_{2,4}$, in order to be in $S^2\setminus Q_{2,4}$, we need the part of $Q_{1,3,4}$ in the hemisphere defined by $\vec{v}_2$ in the direction $-\vec{v}_2$. The crossing of  $\vec{u}_1$ and $\vec{u}_3$ must occur outside of $Q_{1,4}$ and $Q_{1,4}^A$ Thus, their crossing will occur in $Q_{1,3}\setminus Q_{1,3,4}$. Similarly, the crossing of $\vec{n}_1,\vec{n}_3$ will cross inside $Q_{1,4}\setminus Q_{1,3,4}$. $\vec{v}_3$ goes through both of these crossing points thus $\vec{v}_3$ intersects the interior of  $Q_{1,3,4}$.  To be in the region $C_{e_4,e_3}$ (in order to avoid projections of the form $B_i\prime$), we are interested in the  hemisphere defined by the great circle with normal vector $\vec{v}_3$ in the direction of $\vec{v}_3$. Taking all this into account, $Q$ will be either equal to $(\vec{v}_3,-\vec{v}_2,\vec{n}_1,-\vec{u}_2)$ or to $(-\vec{v}_2,-\vec{u}_2,\vec{v}_3)$, depending on whether the crossing of $\vec{v}_2$ with $\vec{u}_2$ occurs inside or outside $Q_{1,3,4}$. 
Thus, we have shown that, if $Q\neq0$, then $A(Q)=A(\vec{v}_3,-\vec{v}_2,\vec{n}_1,-\vec{u}_2)$, if $c_{4,1}=(\vec{p}_{4,1}\cdot\vec{n}_1)\epsilon_{1,3}>0$, and $A(Q)=A(-\vec{v}_2,-\vec{u}_2,\vec{v}_3)$, otherwise.

If $Q=\emptyset$, then we check for case B(ii), by repeating the same algorithm for the walk with reversed orientation.

\end{proof}

\begin{remark}
For the case of $n=3$ and $n=4$, the bracket polynomial of planar knotoids coincides with the bracket polynomial of knotoids in $S^2$, since the planar knotoids that occur are exactly the same as the knotoids in $S^2$.
\end{remark}

\begin{theorem}\label{fullbracket}
Let $E_4$ denote a polygonal chain of 4 edges, $e_1,e_2,e_3,e_4$ in 3-space, then the bracket polynomial of $E_4$ is

\begin{equation}\label{jn4}
\begin{split}
\langle E_4\rangle=& p_{k21}\langle k2.1\rangle+p_{k0,\epsilon_{2,4}}(-A^3)^{\epsilon_{2,4}}+p_{k0,-\epsilon_{2,4}}(-A^3)^{-\epsilon_{2,4}}+p_{k0,-2\epsilon_{2,4}}(-A^3)^{-2\epsilon_{2,4}}\\
&+p_{k0,2\epsilon_{2,4}}(-A^3)^{2\epsilon_{2,4}}+p_{k0,0}\nonumber\\
\end{split}
\end{equation}

\noindent where the coefficients are: 

\begin{align}
p_{k21}=P(K((E_4)_{\xi})=k2.1,wr(E_4)_{\xi})=-\epsilon_{2,4})=
\begin{cases}
\frac{1}{2\pi}A(Q), \text{ if } \epsilon_{1,3}=\epsilon_{1,4}\\
0, \text{ otherwise}\\
\end{cases}
\end{align}

\begin{align}
&p_{k0,\epsilon_{2,4}}=P(K((E_4)_{\xi})=k0,wr((E_4)_{\xi})=\epsilon_{2,4})\nonumber\\
&=
\begin{cases}
2|L(e_2,e_4)|-\frac{1}{2\pi}A(Q_{4,2,1}), \epsilon_{1,3}=\epsilon_{1,4}\\
2|L(e_2,e_4)|+2|L(e_1,e_4)|-\frac{1}{2\pi}(A(Q_{4,2,1})+A(Q_2)+A(Q_1)), \text{ if }\epsilon_{2,4}=\epsilon_{1,4}=-\epsilon_{1,3}\\
2|L(e_2,e_4)|+2|L(e_1,e_3)|-\frac{1}{2\pi}(A(Q_{4,2,1})+A(Q_1)), \text{ if } \epsilon_{2,4}=\epsilon_{1,3}=-\epsilon_{1,4}\\
\end{cases}
\end{align}

\begin{align}
&p_{k0,-\epsilon_{2,4}}=P(K((E_4)_{\xi})=k0,wr((E_4)_{\xi})=-\epsilon_{2,4})\nonumber\\
&=
\begin{cases}
2|L(e_1,e_3)|+2|L(e_1,e_4)|-\frac{1}{2\pi}(A(Q_{1,3,4})+A(Q_2)+A(Q_1)),  \epsilon_{1,3}=\epsilon_{1,4}\\
2|L(e_1,e_3)|-\frac{1}{2\pi}A(Q_{1,3,4}), \text{ if } \epsilon_{2,4}=\epsilon_{1,4}=-\epsilon_{1,3}\\
2|L(e_1,e_4)|-\frac{1}{2\pi}(A(Q_{1,3,4})+A(Q_2)), \text{ if }  \epsilon_{2,4}=\epsilon_{1,3}=-\epsilon_{1,4}\\
\end{cases}
\end{align}

\begin{align}
p_{k0,2\epsilon_{2,4}}=P(K((E_4)_{\xi})=k0,wr((E_4)_{\xi})=2\epsilon_{2,4})=
\begin{cases}
\frac{1}{2\pi}(A(Q_2)-A(Q)), \text{ if }  \epsilon_{2,4}=\epsilon_{1,4}=-\epsilon_{1,3}\\
0, \text{ otherwise}\\
\end{cases}
\end{align}

\begin{align}
&p_{k0,-2\epsilon_{2,4}}=\nonumber\\
&P(K((E_4)_{\xi})=k0,wr((E_4)_{\xi})=-2\epsilon_{2,4})=
\begin{cases}
\frac{1}{2\pi}(A(Q_1)-A(Q)), \text{ if }  \epsilon_{1,4}=\epsilon_{1,3}=-\epsilon{2,3}\\
0, \text{ otherwise}\\
\end{cases}
\end{align}

\noindent and

\begin{equation}
\begin{split}
&p_{k0,0}=P(K((E_4)_{\xi})=k0,wr((E_4)_{\xi})=0)=1-(P(K((E_4)_{\xi})=k0,wr((E_4)_{\xi})=-2\epsilon_{2,4})\\
&+P(K((E_4)_{\xi})=k0,wr((E_4)_{\xi})=2\epsilon_{2,4})+P(K((E_4)_{\xi})=k0,wr((E_4)_{\xi})=-\epsilon_{2,4})\\
&+P(K((E_4)_{\xi})=k0,wr((E_4)_{\xi})=\epsilon_{2,4})+P(K((E_4)_{\xi})=k2.1,wr(E_4)_{\xi})=-\epsilon_{2,4})\\
\end{split}
\end{equation}

\noindent where $\epsilon_{i,j}$ denotes the sign of the linking number between $e_i,e_j$, $Q_1=Q_{1,3,4}\setminus Q_{2,4}$, $Q_2=Q_{4,2,1}\setminus Q_{1,3}$ and $Q=Q((E_4)_{\vec{\xi}}=k2.1)$. $P(Q)$ is derived in Theorem \ref{qijj+1} and $Q_1$ is shown in Table \ref{tabletheorem}. $Q_{4,2,1}$, $Q_2$ are derived with the same formulas for the reversed polygonal chain.

\end{theorem}

\begin{table}
\centering
\begin{tabular}{|l|l|l|}
\hline
$\epsilon_{1,3}=\epsilon_{1,4}$, $w<0,w_0<0$        & $Q_{1,3,4}$                 & $Q_1$ \\
\hline
$c_{3,1}>0,c_{4,1}>0,c_{3,0}>0,c_{4,0}>0$               & $(\vec{n}_4,\vec{n}_1,-\vec{u}_2,\vec{n}_3)$        & $(\vec{n}_4,-\vec{v}_3,-\vec{u}_2,\vec{n}_3)\cup Q$ \\
$c_{3,1}>0,c_{4,1}>0,c_{3,0}<0,c_{4,0}<0$               & $(\vec{n}_4,\vec{n}_1,-\vec{u}_2,-\vec{u}_1)$         & $(\vec{n}_4,-\vec{v}_3,-\vec{u}_2,-\vec{u}_1)\cup Q$\\
$c_{3,1}>0,c_{4,1}>0,c_{3,0}>0,c_{4,0}<0$               & $(\vec{n}_4,\vec{n}_1,-\vec{u}_2,-\vec{u}_1,\vec{n}_3)$     & $(\vec{n}_4,-\vec{v}_3,-\vec{u}_2,-\vec{u}_1,\vec{n}_3)\cup Q$\\
$c_{3,1}>0,c_{4,1}>0,c_{3,0}<0,c_{4,0}>0$               & $(\vec{n}_4,\vec{n}_1,-\vec{u}_2,\vec{n}_3,-\vec{u}_1)$     & $(\vec{n}_4,-\vec{v}_3,-\vec{u}_2,\vec{n}_3,-\vec{u}_1)\cup Q)$\\
$c_{3,1}<0,c_{4,1}<0,c_{3,0}>0,c_{4,0}>0$               & $(\vec{n}_4,-\vec{u}_3,-\vec{u}_2,\vec{n}_3)$       & $(\vec{n}_4,-\vec{v}_3,-\vec{u}_2,\vec{n}_3)\cup Q$\\
$c_{3,1}<0,c_{4,1}<0,c_{3,0}<0,c_{4,0}<0$               & $(\vec{n}_4,-\vec{u}_3,-\vec{u}_2,-\vec{u}_1)$        & $(\vec{n}_4,-\vec{v}_3,-\vec{u}_2,-\vec{u}_1) \cup Q$ \\ 
$c_{3,1}<0,c_{4,1}<0,c_{3,0}>0,c_{4,0}<0$               & $(\vec{n}_4,-\vec{u}_3,-\vec{u}_2,-\vec{u}_1,\vec{n}_3)$    & $(\vec{n}_4,-\vec{v}_3,-\vec{u}_2,-\vec{u}_1,\vec{n}_3)\cup Q$\\
$c_{3,1}<0,c_{4,1}<0,c_{3,0}<0,c_{4,0}>0$               & $(\vec{n}_4,-\vec{u}_3,-\vec{u}_2,\vec{n}_3,-\vec{u}_1)$    & $(\vec{n}_4,-\vec{v}_3,-\vec{u}_2,\vec{n}_3,-\vec{u}_1)\cup Q$\\
$c_{3,1}>0,c_{4,1}<0,c_{3,0}>0,c_{4,0}>0$               & $(\vec{n}_4,\vec{n}_1,-\vec{u}_3,-\vec{u}_2,\vec{n}_3)$    & $(\vec{n}_4,-\vec{v}_3,-\vec{u}_2,\vec{n}_3)\cup Q$\\
$c_{3,1}>0,c_{4,1}<0,c_{3,0}<0,c_{4,0}<0$               & $(\vec{n}_4,\vec{n}_1,-\vec{u}_3,-\vec{u}_2,-\vec{u}_1)$     & $(\vec{n}_4,-\vec{v}_3,-\vec{u}_2,-\vec{u}_1)\cup Q$\\
$c_{3,1}>0,c_{4,1}<0,c_{3,0}>0,c_{4,0}<0$               & $(\vec{n}_4,\vec{n}_1,-\vec{u}_3,-\vec{u}_2,-\vec{u}_1,\vec{n}_3)$ & $(\vec{n}_4,-\vec{v}_3,-\vec{u}_2,-\vec{u}_1,\vec{n}_3)\cup Q$\\
$c_{3,1}>0,c_{4,1}<0,c_{3,0}<0,c_{4,0}>0$               & $(\vec{n}_4,\vec{n}_1,-\vec{u}_3,-\vec{u}_2,\vec{n}_3,-\vec{u}_1)$ & $(\vec{n}_4,-\vec{v}_3,-\vec{u}_2,\vec{n}_3,-\vec{u}_1)\cup Q$\\
$c_{3,1}<0,c_{4,1}>0,c_{3,0}>0,c_{4,0}>0$               & $(\vec{n}_4,-\vec{u}_3,\vec{n}_1,-\vec{u}_2, \vec{n}_3)$   & $(\vec{n}_4,-\vec{v}_3,-\vec{u}_2, \vec{n}_3)\cup Q$ \\
$c_{3,1}<0,c_{4,1}>0,c_{3,0}<0,c_{4,0}<0$               & $(\vec{n}_4,-\vec{u}_3,\vec{n}_1,-\vec{u}_2, -\vec{u}_1)$    & $(\vec{n}_4,-\vec{v}_3,-\vec{u}_2, -\vec{u}_1)\cup Q$\\
$c_{3,1}<0,c_{4,1}>0,c_{3,0}>0,c_{4,0}<0$               & $(\vec{n}_4,-\vec{u}_3,\vec{n}_1,-\vec{u}_2,-\vec{u}_1,\vec{n}_3)$ & $(\vec{n}_4,-\vec{v}_3,-\vec{u}_2,-\vec{u}_1,\vec{n}_3)\cup Q$\\
$c_{3,1}<0,c_{4,1}>0,c_{3,0}<0,c_{4,0}>0$               & $(\vec{n}_4,-\vec{u}_3,\vec{n}_1,-\vec{u}_2,\vec{n}_3,-\vec{u}_1)$ & $(\vec{n}_4,-\vec{v}_3,-\vec{u}_2,\vec{n}_3,-\vec{u}_1)\cup Q$\\
\hline
$\epsilon_{1,3}=\epsilon_{1,4}$, $w>0$ or $w_0>0$        & $Q_{1,3,4}$                 & $Q_1$  \\
\hline
       & $\emptyset$                 & $\emptyset$  \\
\hline
$\epsilon_{1,3}=-\epsilon_{1,4},w<0$        & $Q_{1,3,4}$                 & $Q_1$   \\
 \hline
 $c_{4,0}>0,c_{4,1}>0$               & $(\vec{n}_2,-\vec{u}_1,-\vec{u}_2,-\vec{u}_3)$              &  $Q_{1,3,4}\setminus(\vec{v}_1,\vec{v}_2,\vec{v}_3,\vec{n}_2)$\\
 $c_{4,0}<0,c_{4,1}<0$               & $(\vec{n}_2,\vec{n}_1,-\vec{u}_2,\vec{n}_3)$        &    $Q_{1,3,4}$\\
$c_{4,0}<0,c_{4,1}>0$               & $(\vec{n}_2,\vec{n}_1,-\vec{u}_2,-\vec{u}_3)$       &   $Q_{1,3,4}\setminus(\vec{v}_1,\vec{v}_2,\vec{n}_1,\vec{n}_2)$\\
$c_{4,0}>0,c_{4,1}<0$               &  $(\vec{n}_2,-\vec{u}_1,-\vec{u}_2,\vec{n}_3)$      &   $Q_{1,3,4}$\\
\hline
$\epsilon_{1,3}=-\epsilon_{1,4},w>0$        & $Q_{1,3,4}$                 & $Q_1$   \\
 \hline
$c_{4,0}>0,c_{4,1}>0$               & $(\vec{n}_2,-\vec{u}_1,\vec{n}_4,-\vec{u}_3)$       &     $Q_{1,3,4}\setminus(-\vec{u}_3,\vec{n}_4,\vec{v}_3,\vec{n}_2)$\\
$c_{4,0}<0,c_{4,1}<0,c_{4\prime,1\prime}>0$               & $(\vec{n}_2,\vec{n}_1,\vec{n}_4,\vec{n}_3)$              &      $Q_{1,3,4}\setminus(\vec{v}_3,-\vec{v}_2,\vec{n}_2,\vec{n}_1,\vec{n}_4)$\\
$c_{4,0}<0,c_{4,1}<0,c_{4\prime,1\prime}<0$               & $(\vec{n}_2,\vec{n}_1,\vec{n}_4,\vec{n}_3)$              &      $Q_{1,3,4}\setminus(\vec{v}_3,-\vec{v}_2,\vec{n}_1,\vec{n}_4)$\\
$c_{4,0}<0,c_{4,1}>0$               & $(\vec{n}_2,\vec{n}_1,\vec{n}_4,-\vec{u}_3)$        &     $\emptyset$\\
$c_{4,0}>0,c_{4,1}<0$               & $(\vec{n}_2,-\vec{u}_1,\vec{n}_4,\vec{n}_3)$        &     $Q_{1,3,4}$\\
\hline
\end{tabular}
\caption{The spherical polygons $Q_{1,3,4}$ and $Q_1=Q_{1,3,4}\setminus Q_{2,4}$, respectively, are computed by using the above expressions. The expression $(\vec{w}_1,\dotsc,\vec{w}_n)$ denotes the spherical polygon defined by the intersection of the great circles with normal vectors $\vec{w}_1,\dotsc,\vec{w}_n$ in the counterclockwise orientation (see Definition \ref{defn1}).  The expressions depend on the conformation of the chain in 3-space, where $c_{3,1}=(\vec{p}_{3,1}\cdot\vec{n}_1)_{\epsilon_{1,3}}$, $c_{4,1}=(\vec{p}_{4,1}\cdot\vec{n}_1)\epsilon_{1,3}$, $c_{3,0}=(\vec{p}_{3,0}\cdot\vec{n}_3)\epsilon_{1,3}$, $c_{4,0}=(\vec{p}_{4,0}\cdot\vec{n}_3)\epsilon_{1,3}$, $c_{4\prime,1\prime}=(\vec{p}_{1,4}\cdot(-\vec{v}_2))\epsilon_{2,4}$ and $w=(\vec{u}_2\times (-\vec{n}_2))\cdot(\vec{u}_2\times \vec{n}_4)$, where $\vec{n}_1,\vec{u}_i,\vec{v}_i$ are the normal vectors to the quadrilaterals $T_{1,3},T_{1,4},T_{2,4}$ and where $\vec{p}_{i,j}$ is the vector that connects vertex $i$ to vertex $j$ in 3-space. The areas of $Q_{4,2,1}$ and $Q_2$ are obtained from the areas $Q_{1,3,4}$ and $Q_1$ of the polygonal chain with reversed orientation (see proof of Theorem \ref{fullbracket}).}\label{tabletheorem}
\end{table}

\begin{proof}

In the following, for simplicity, we will write $P(A_1)$ to express the probability  $P(K((E_4)_{\xi})=k0_{A_1})$, etc.

By Proposition \ref{number4}, $k2.1$ is a possible knotoid diagram only when $\epsilon_{1,3}=\epsilon_{1,4}$, in which case, it also implies that $\epsilon_{2,4}=-\epsilon_{1,3}$. The probability of obtaining $k2.1$ is found in Theorem \ref{knotoid4}. 

Thus, we only need to examine the probabilities of obtaining the trivial knotoid with a given writhe. By inspection of the diagrams shown in Figure \ref{fig:4knotoids}, we first notice the following:

\begin{align}
P(K((E_4)_{\xi})=k0,wr((E_4)_{\xi})=\epsilon_{2,4})=
\begin{cases}
P(A_2), \text{ if }  \epsilon_{1,3}=\epsilon_{1,4}=-\epsilon_{2,4}\\
P(A_3)+P(A_2)+P(C), \text{ if }  \epsilon_{2,4}=\epsilon_{1,4}=-\epsilon_{1,3}\\
P(A_1)+P(A_2)+P(C), \text{ if }  \epsilon_{2,4}=\epsilon_{1,3}=-\epsilon_{1,4}\\
\end{cases}
\end{align}

\begin{align}
P(K((E_4)_{\xi})=k0,wr((E_4)_{\xi})=-\epsilon_{2,4})=
\begin{cases}
P(A_1)+P(A_3)+P(C), \text{ if }  \epsilon_{1,3}=\epsilon_{1,4}=-\epsilon_{2,4}\\
P(A_1), \text{ if }  \epsilon_{2,4}=\epsilon_{1,4}=-\epsilon_{1,3}\\
P(A_3), \text{ if }  \epsilon_{2,4}=\epsilon_{1,3}=-\epsilon_{1,4}\\
\end{cases}
\end{align}

\begin{align}
P(K((E_4)_{\xi})=k0,wr((E_4)_{\xi})=2\epsilon_{2,4})=
\begin{cases}
P(B_{ii})+P(B_{ii}\prime), \text{ if }  \epsilon_{2,4}=\epsilon_{1,4}=-\epsilon_{1,3}\\
0, \text{ otherwise}\\
\end{cases}
\end{align}

\begin{align}
P(K((E_4)_{\xi})=k0,wr((E_4)_{\xi})=-2\epsilon_{2,4})=
\begin{cases}
P(B_{i})+P(B_{i}\prime), \text{ if }  \epsilon_{1,3}=-\epsilon_{1,4}=-\epsilon_{2,4}\\
0, \text{ otherwise}\\
\end{cases}
\end{align}

We will compute these probabilities in the three cases: $\epsilon_{1,3}=\epsilon_{1,4}$,  $\epsilon_{1,3}=-\epsilon_{1,4}=\epsilon_{2,4}$,  $\epsilon_{1,3}=-\epsilon_{1,4}=-\epsilon_{2,4}$.

First, we notice that, in all cases, due to the connectivity of the chain,

\begin{align}\label{subsets}
&Q_{2,4}\cap Q_{1,3}\subset Q_{2,4}\cap Q_{1,4}=Q_{4,2,1}\\
&Q_{2,4}\cap Q_{1,3}\subset Q_1=Q_{1,3}\cap Q_{1,4}=Q_{1,3,4}\nonumber
\end{align}

The probabilities can be expressed as:

\begin{align}
&P(A_1)=2|L(e_1,e_3)|-\frac{1}{2\pi}A(Q_{1,3,4})\nonumber\\
&P(A_2)=2|L(e_1,e_3)|-\frac{1}{2\pi}A(Q_{4,2,1})\nonumber\\
&P(A_3)=2|L(e_1,e_4)|-\frac{1}{2\pi}A(Q_{4,2,1}\setminus Q_{1,3})-A(Q_{1,3,4})\nonumber\\
&P(C)=\frac{1}{2\pi}A(Q_{2,4}\cap Q_{1,3,4})\nonumber\\
&P(B_i)+P(B_i\prime)=\frac{1}{2\pi}A(Q_{1,3,4}\setminus Q_{2,4})\nonumber\\
&P(B_{ii})+P(B_{ii}\prime)=\frac{1}{2\pi}A(Q_{4,2,1}\setminus Q_{1,3})\nonumber
\end{align}

From all these equations, and using the notation $Q_1=Q_{1,3,4}\setminus Q_{2,4}$ and $Q_2=Q_{4,2,1}\setminus Q_{1,3}$, we obtain the expressions of the statement of the Theorem.  

We proceed with finding finite forms for $Q_{1,3,4}$ and $Q_1$ from which the finite forms of $Q_{4,2,1}$ and  $Q_2$ are also derived.

\noindent\textit{Finite form of $Q_{1,3,4}$} 

The finite form of $Q_{1,3,4}$ is found by Theorem \ref{qijj+1} for $i=0,j=2$.

\noindent\textit{Finite form of $Q_{4,2,1}$}: 

For the finite form of $Q_{4,2,1}$ we think as follows: Let $R(E_4)$ to denote the polygonal chain $E_4$ with reversed numbering of vertices. Let us denote its edges $e_1^{\prime},e_2^{\prime},e_3^{\prime},e_4^{\prime}$. Then $Q_{4,2,1}=Q_{1^{\prime},3^{\prime},4^{\prime}}$. This can be obtained from table \ref{tabletheorem} determined by the algorithm described in Section \ref{Gaussfinite}  for $n_i\prime,u_i\prime$ which are related to the normal vectors of $E_4$ as follows: $\vec{n}_1\prime=-\vec{v}_2$, $\vec{n}_2\prime=-\vec{v}_1$, $\vec{n}_3\prime=-\vec{v}_4$, $\vec{n}_4\prime=-\vec{v}_3$, $\vec{u}_1\prime=-\vec{u}_2$, $\vec{u}_2\prime=-\vec{u}_1$, $\vec{u}_3\prime=-\vec{u}_4$, $\vec{u}_4\prime=-\vec{u}_3$. Accordingly, $w\prime=(\vec{u}_1\times\vec{v}_1)\cdot(\vec{u}_1\times(-\vec{v}_3))$, $w_0\prime=(\vec{n}_4\times\vec{v}_2)\cdot(\vec{n}_4\times\vec{n}_3)$, $\epsilon_{1\prime,3\prime}=\epsilon_{2,4}$ and $\epsilon_{1\prime,4\prime}=\epsilon_{1,4}$. Finally, $c_{3\prime,1\prime}=(\vec{p}_{3\prime,1\prime}\cdot\vec{n}_1\prime)\epsilon_{1\prime,3\prime}=(\vec{p}_{1,3}\cdot(-\vec{v}_2)\epsilon_{2,4}$, otherwise  $c_{4\prime,1\prime}=(\vec{p}_{4\prime,1\prime}\cdot\vec{n}_1\prime)\epsilon_{1\prime,3\prime}=(\vec{p}_{1,4}\cdot(-\vec{v}_2))\epsilon_{2,4}$, $c_{3\prime,0\prime}=(\vec{p}_{3\prime,0\prime}\cdot\vec{n}_3\prime)\epsilon_{1\prime,3\prime}=(\vec{p}_{1,4}\cdot(-\vec{v}_4))\epsilon_{2,4}$, otherwise  $c_{4\prime,0\prime}=(\vec{p}_{4\prime,0\prime}\cdot\vec{n}_3\prime)\epsilon_{1\prime,3\prime}=(\vec{p}_{0,4}\cdot(-\vec{v}_4))\epsilon_{2,4}$, when $\epsilon_{1\prime,3\prime}=\epsilon_{1\prime,4\prime}$ and 
$c_{4\prime,0\prime}=(\vec{p}_{4\prime,0\prime}\cdot\vec{n}_1\prime)\epsilon_{1\prime,3\prime}=(\vec{p}_{0,4}\cdot(-\vec{v}_2))\epsilon_{2,4}$, otherwise  $c_{4\prime,1\prime}=(\vec{p}_{4\prime,1\prime}\cdot\vec{n}_3\prime)\epsilon_{1\prime,3\prime}=(\vec{p}_{0,3}\cdot(-\vec{v}_4))\epsilon_{2,4}$,  when $\epsilon_{1\prime,3\prime}=-\epsilon_{1\prime,4\prime}$

\noindent\textit{Finite form of $Q_1$} 

\noindent - Case $\epsilon_{1,4}=\epsilon_{1,3}=-\epsilon_{2,4}$: 
One can derive from the proof of Theorem \ref{knotoid4} the area of $Q_{1,3,4}\setminus Q_{2,4}$. The area will be $Q_1=Q\cup(\vec{n}_4,-\vec{v}_3,-\vec{u}_2,x)$,  where $x$ is equal to $-\vec{u}_1$ or $\vec{n}_3$ or $\vec{n}_3,-\vec{u}_1$ or $-\vec{u}_1,\vec{n}_3$, depending on the signs of $c_{0,3},c_{0,4}$ (see Table \ref{tabletheorem}).

Next, we consider the case $\epsilon_{1,4}=-\epsilon_{1,3}$ and refer to Figure \ref{fig:case43}  as an illustrative example. Since $\vec{u}_3=-\vec{v}_1$ and $\vec{n}_3=\vec{v}_4$, these spherical edges (which bound $Q_{2,4}$) do not cross the interior of $Q_{1,3,4}$. In order to find $Q_1=Q_{1,3,4}\setminus Q_{2,4}$, we examine if and how $\vec{v}_2$ and $\vec{v}_3$ intersect the interior of $Q_{1,3,4}$.  Figure \ref{fig:case43} shows the relative positions of $\vec{v}_1,\vec{v}_4,\vec{v}_2$ determined by the connectivity of the chain and the orientations of $\vec{v}_1,\vec{v}_4$ are also given by the known orientations of $\vec{u}_3$ and $\vec{n}_3$. 

\noindent - Case $\epsilon_{1,4}=-\epsilon_{1,3}=\epsilon_{2,4}$: (This is the case where $c_{4,1}<0$ in Table \ref{tabletheorem}). This corresponds to the case where $\epsilon_{1\prime,4\prime}=\epsilon_{1\prime,3\prime}$ for the reversed walk. First of all, in this case, we notice that when $c_{4,0}>0$, then $w\prime>0$ and, similarly, when $w<0$ then $w_0\prime>0$, thus in these cases $Q_{4,2,1}=\emptyset$, giving $Q_1=Q_{1,3,4}$. Thus, the only case that might give $Q_{2,4}\cap Q_{1,3,4}\neq\emptyset$ is the case $w>0, c_{4,0}<0$, equivalently, $w>0,w_0<0$, (see Figure \ref{fig:case43}). In that case the great circle with normal vector $\vec{v}_3$ intersects the interior of $Q_{1,3,4}$ (since the face with normal vector $\vec{v}_3$ is in-between the faces with normal vectors $\vec{n}_1,\vec{n}_3$). To examine the intersection of $Q_{2,4}\cap Q_{1,3,4}$, we examine the reversed oriented polygon, $R(E_4)$ (see previous paragraph). The above conditions correspond to the case where $\epsilon_{1\prime,3\prime}=\epsilon_{1\prime,4\prime}$, $w\prime<0,w_0\prime<0$, which is the case that can give the non-trivial knotoid. Thus, using Theorem \ref{knotoid4}, we derive that for $w\prime<0$, if $c_{4\prime,1\prime}>0$, then $Q_1=Q_{1,3,4}\setminus(v_3,-v_2,n_2,n_1,n_4)$ and if $c_{4\prime,1\prime}<0$, then $Q_1=Q_{1,3,4}\setminus(v_3,-v_2,n_1,n_4)$.

\noindent - Case $\epsilon_{1,4}=-\epsilon_{1,3}=-\epsilon_{2,4}$: (This is the case where $c_{4,1}>0$ in Table \ref{tabletheorem}) As in the previous case, in order to find $Q_1=Q_{1,3,4}\setminus Q_{2,4}$, we need the area of $Q_{1,3,4}$ that is determined by the great circles $\vec{v}_2$ and $\vec{v}_3$. To find these intersections, we will examine $Q_{4,2,1}$ using the reverse walk  with $\epsilon_{1\prime,3\prime}=-\epsilon_{1\prime,4\prime}$, and we notice that in all cases, $c_{1\prime,4\prime}=(p_{4\prime,1\prime}\cdot\vec{n}_3\prime)\epsilon_{1\prime,3\prime}=(p_{0,3}\cdot(-\vec{v}_4))\epsilon_{2,4}=(p_{0,3}\cdot(-\vec{n}_3))\epsilon_{1,3}>0$. Indeed, since $\vec{n}_3$ is the normal vector to the face defined by the vertices 1,2,3, of the tetrhedral $T_{1,4}$ and points inwards if $\epsilon_{1,3}>0$ (in the direction of vertex 3) or outwards otherwise.  Thus $c_{1\prime,4\prime}>0$ in all cases. Thus, the intersection will depend on the sign of $c_{0\prime,4\prime}=(p_{4\prime,0\prime}\cdot\vec{n}_1\prime)\epsilon_{1\prime,3\prime}=(p_{0,4}\cdot(-\vec{v}_2))\epsilon_{2,4}$. This sign will depend on the sign of $c_{4,0}=(\vec{p}_{4,0}\cdot\vec{n}_1)\epsilon_{1,3}$ and the sign of $w$, which determines if $\vec{u}_2$ lies between $\vec{n}_2,\vec{n}_4$.

If $c_{4,0}<0$ then $w\prime<0$ since we can verify that the face with normal vector $\vec{u}_1$ is between the faces with normal vectors $\vec{v}_1,\vec{v}_3$, and $w\prime>0$ if $c_{4,0}>0$. If $w<0$ then $c_{0\prime,4\prime}=(p_{4\prime,0\prime}\cdot\vec{n}_1\prime)\epsilon_{1\prime,3\prime}=(p_{0,4}\cdot(-\vec{v}_2))\epsilon_{2,4}<0$ since $\vec{v}_2$ points in the opposite direction of the region that contains the vertex 0 when $\epsilon_{2,4}<0$, and $c_{0\prime,4\prime}>0$ if $w>0$.

Thus, by using Table \ref{tabletheorem} for the reversed walk we find that if $c_{4,0}<0$ and  $w<0$, then  $Q_1=Q_{1,3,4}\setminus(v_1,v_2,n_1,n_2)$.
If $c_{4,0}<0$ and $w>0$, then $Q_1=\emptyset$.
If $c_{4,0}>0$ and $w<0$, then  $Q_1=Q_{1,3,4}\setminus(v_1,v_2,v_3,n_2)$.
If $c_{4,0}>0$ and $w>0$,  then $Q_1=Q_{1,3,4}\setminus(v_1,n_4,v_3,n_2)$.

\begin{figure}
   \begin{center}
     \includegraphics[width=0.9\textwidth]{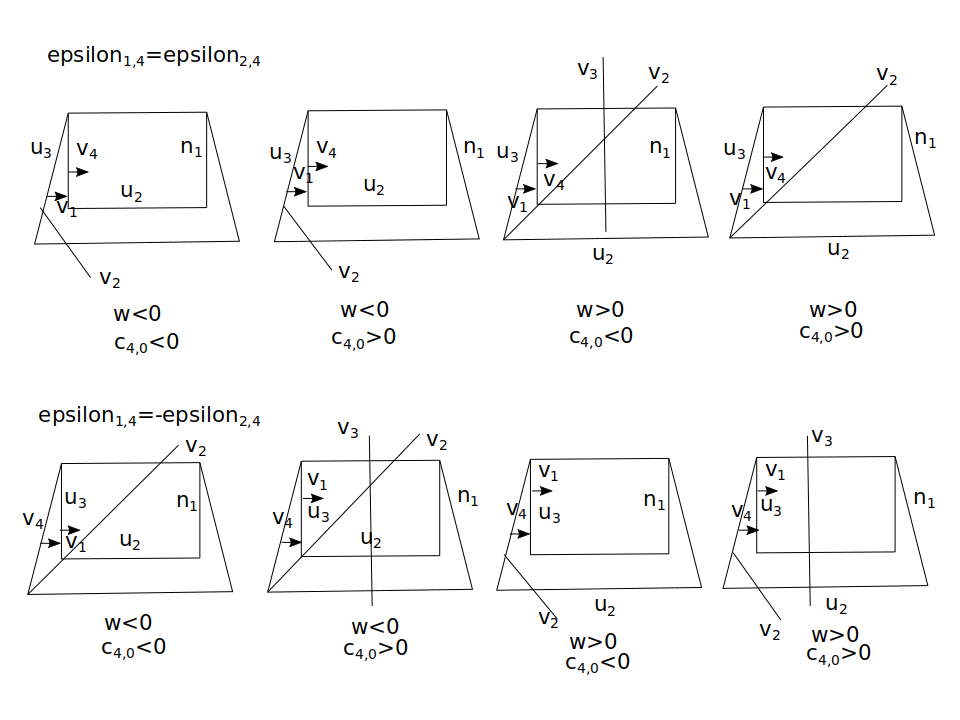}
     \caption{Representation of $Q_{1,3,4}$ when $\epsilon_{1,3}=\epsilon_{2,4}$. In this case, one great circle of the boundary of $Q_{1,3,4}$ is the one with normal vector $\vec{n}_2$ (top boundary in the figure). The lower great circle (bottom boundary) is $\vec{u}_2$ or $\vec{n}_4$, depending on whether $\epsilon_{1,4}=\epsilon_{2,4}$ or not (equivalently, depending on the sign of $c_{1,4}$). Similar considerations define the other boundaries, where  $c_{4,0}=(\vec{p}_{4,0}\cdot\vec{n}_1)\epsilon_{1,3}$, $w=(\vec{u}_2\times (-\vec{n}_2))\cdot(\vec{u}_2\times \vec{n}_4)$, $\vec{n}_3=\vec{v}_4$ and $\vec{u}_3=-\vec{v}_1$. To determine $Q_1=Q_{1,3,4}\setminus Q_{2,4}$, we examine how $\vec{v}_2$ and $\vec{v}_3$ intersect $Q_{1,3,4}$ (see proof of Theorem \ref{fullbracket}). The results are shown in Table \ref{tabletheorem}. }
     \label{fig:case43}
   \end{center}
\end{figure}

\noindent\textit{Finite form of $Q_2$}: 

For the finite form of $Q_2$ we think as follows: Let $R(E_4)$ to denote the polygonal chain $E_4$ with reversed numbering of vertices as described in the Finite form of $Q_{4,2,1}$. Then $Q_2=Q_{4,2,1}\setminus Q_{1,3}=Q_{1^{\prime},3^{\prime},4^{\prime}}\setminus Q_{2^{\prime},4^{\prime}}=Q_1^{\prime}$, which is found earlier.

\end{proof}

\section{A finite form for the Jones polynomial of an open polygonal curve with 3 and 4 edges}\label{Jonesfinite}

To find a finite form of the Jones polynomial, we first find a finite form for the normalized bracket polynomial.

Notice that the case of closed curves is reduced to the Jones polynomial of any projection of the closed chain. Thus, here we focus on the open case where the average over all projections is needed.

In the case of a polygonal chain with 3 edges, we denote $E_3$, the Jones polynomial is always trivial.
Let $E_4$ denote a polygonal chain of 4 edges. Then, by Propositions \ref{number4} and \ref{case}, the only non-trivial bracket polynomial is $k2.1$ and the writhe of the diagram is either 2 or -2. Thus the Jones polynomial of $E_4$ has the following form:

\begin{equation}
    \begin{split}
&f(E_4)= P(K((E_4)_{\xi})=k2.1)((-A)^3)^{-(\pm2)}\langle k2.1\rangle+\sum_{j=-2}^2P(K((E_4)_{\xi})=k0,wr((E_4)_{\xi})=j)\cdot1\nonumber\\
&= P(K((E_4)_{\xi})=k2.1)((-A)^3)^{-(\pm2)}\langle k2.1\rangle+(1-P(K((E_4)_{\xi})=k2.1))\nonumber\\ 
\end{split}
\end{equation}

\noindent where  $P(K((E_4)_{\xi})=k2.1)$ is  defined in Theorem \ref{knotoid4}.

\begin{corollary}
Let $E_4$ denote a polygonal chain of 4 edges, $e_1,e_2,e_3,e_4$ in 3-space, then the normalized bracket polynomial of $E_4$ is

\begin{equation}\label{jn4}
\begin{split}
&f(E_4)= P(K((E_4)_{\xi})=k2.1)((-A)^3)^{-2\epsilon_{2,4}}\langle k2.1\rangle+(1-P(K((E_4)_{\xi})=k2.1))\\
\end{split}
\end{equation}

\noindent where $P(K((E_4)_{\xi})=k2.1)$ is  defined in Theorem \ref{knotoid4}. 

\end{corollary}

\noindent\textbf{Example 2: } As in Example 1, we consider the polygonal chain 

\noindent$I(t)=((0,1,0),(0,0,0),(-0.2,0.8,0.8),(0.1,0.8,-0.8),(0.1+1.2\cos(a+t),0.5,-0.8+1.2\sin(a+t)))$. 

\noindent where $a=32000\pi/100000$ and $t$ is in units of $2\pi/100000$. The chain attains a more compact configuration as time increases.
The bracket polynomials and Jones polynomials of the chain vary in time.

Figure \ref{fig:example2J} shows the Jones polynomials at different times. For comparison, the Jones polynomial of the trefoil knot (above) and of the 2-dimensional knotoid diagram $k2.1$ (below) are shown as well. Notice that a polygonal chain needs at least 6 edges to form a trefoil knot \cite{Calvo2001}. Nevertheless,  in Figure \ref{fig:example2J} above we see a small but continuous change of the polynomial closer to that of the trefoil knot. Indeed, we notice that the tight configuration that attains the open chain in 3-space would be a necessary part of the knotting pathway of the open chain to form a trefoil knot. In Figure \ref{fig:example2J} below (left) we also plot the Jones polynomial of the open chain in 3-space as a function of time and the Jones polynomial of the knotoid diagram of k2.1. We see that the Jones polynomial of the open chain tends to that of the 2-dimensional knotoid k2.1. Indeed, as the configuration tightens, it almost becomes 2-dimensional, giving in most projections the knotoid k2.1. However, it will never be exactly equal to that. In Figure \ref{fig:example2J} below (right) we plot the roots of the Jones polynomial in time and those of the trefoil knot.

Figure \ref{fig:example2K} shows the Kauffman bracket polynomial of the open 3-dimensional chain in time and that of the standard diagram of the knotoid k2.1. Again, we see the open chain tending to that of the diagram, due to the tightening of the configuration.

\begin{remark}
For the computation of the areas of the spherical polygons defined by the ordered normal vectors in Theorem \ref{fullbracket} and Table \ref{tabletheorem}, in the examples reported in this manuscript, the inverse cosine of the dot product of consecutive normal vectors were used to find the interior angles of the spherical polygons. This simple approach works only for convex spherical polygons (which is the case in this example). 
\end{remark}

\begin{figure}[H]
   \begin{center}
   \includegraphics[width=0.8\textwidth]{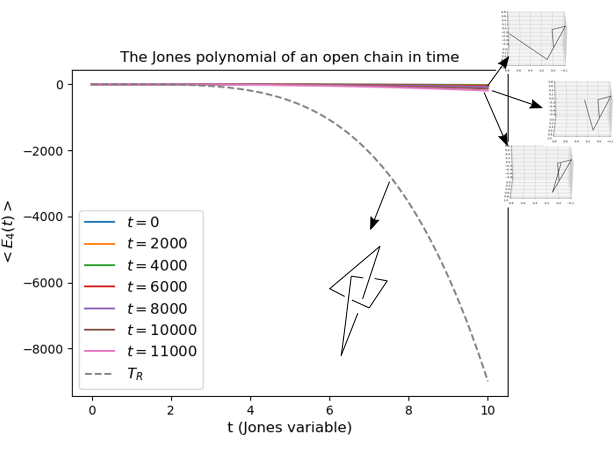}\\
     \includegraphics[width=0.5\textwidth]{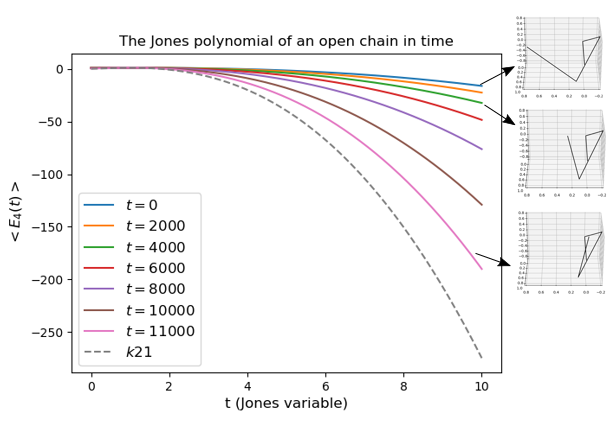}\includegraphics[width=0.5\textwidth]{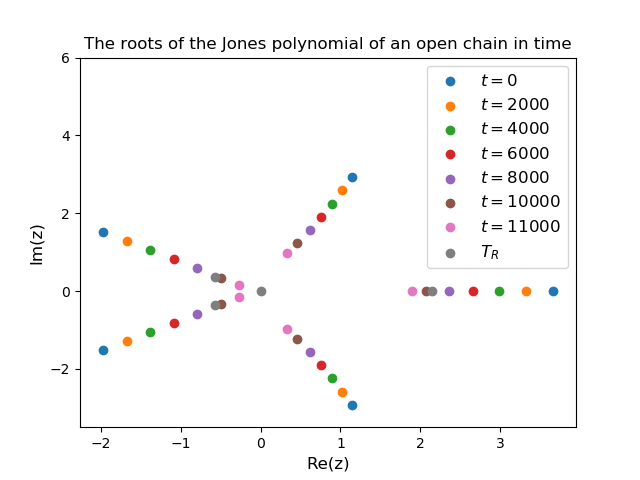}
     \caption{The Jones polynomial of the chain in 3-space as it deforms in time to tighten a compact configuration. Top: The dotted curve shows the Jones polynomial of the trefoil knot (we denote $T_R$). Even though a chain with 4 edges cannot form the trefoil knot \cite{Calvo2001}, we see that the polynomial of the open chain tends to that of the trefoil knot, as this part of the configuration would be a part of the knotting pathway towards a trefoil knot. Bottom: (Left) The dotted curve shows the Jones polynomial of the knotoid $k2.1$ (a 2-dimensional diagram). We see that the chain tightens to a configuration that in most projections will give the knotoid k2.1, which explains why the polynomials tend to that of k2.1. (Right) The roots of the Jones polynomial of the open chain in 3-space as a function of time and the roots of the trefoil polynomial. }
     \label{fig:example2J}
   \end{center}
\end{figure}

\begin{figure}[H]
   \begin{center}
     \includegraphics[width=0.6\textwidth]{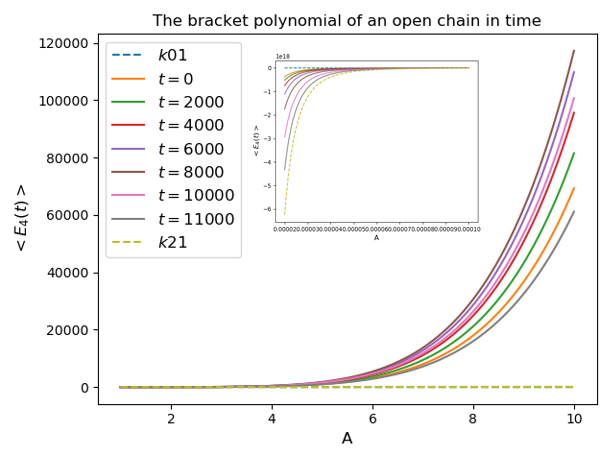}
     \caption{The Kauffman bracket polynomial of an open polygonal chain as it moves in time. The inset plot shows the polynomial for values of the parameter $A$ less than 1.}
     \label{fig:example2K}
   \end{center}
\end{figure}

\section{Conclusions} In this work we defined the Kauffman bracket polynomial and the Jones polynomial in a way that is applicable to both open and closed curves in 3-space. We showed that for open chains these are continuous functions in the space of configurations. In doing this, we introduced a new method of measuring complexity of open chains, that combines the fundamental concepts of the Gauss linking integral and the theory of knotoids. This approach opens a new direction of research in applied knot theory where more of the machinery of knot and link polynomials can be rigorously applied to open chains for the first time. 

Moreover, we showed how these functions of complexity obtain a finite form for polygonal chains. We derived specific finite formulas for the computation of the Kauffman and Jones polynomials in the basic case of a polygonal chain of 4 edges. This study lays the foundation for the derivation of a finite form for a larger number of edges. We stress that the number of edges that are relevant in applications, such as polymers, may not be equal to the exact number of covalent bonds, but rather equal to the number of Kuhn segments, or even less than that, equal to the number of entanglement strands in a primitive path \cite{Tzoumanekas2006,Panagiotou2013b,Panagiotou2011}, for which, even less than 10 edges are relevant. Similarly, proteins may be represented by their sequence of secondary structure elements as building blocks, for which less than 10 edges may also be relevant \cite{Panagiotou2020}.

Even for this small number of edges, our numerical results show that the polynomials are sensitive to the motion of the polygonal chain and indicative of the transition to more compact conformations.  For a larger number of edges these measures will directly reflect the entanglement of the open chain and how knotting occurs.  We stress that these tools can also be applied to collections of open and closed chains and we expect them to have impactful applications.  They allow to be included in formulations of mechanical models of elastic coils \cite{Charles2019,Patil2020}. Also, they allow to accurately described knotting pathways in proteins for the first time \cite{ODonnol2018}. As well as in theories that derive important quantities in polymer physics, such as the entanglement length \cite{Qin2011,Panagiotou2013b,Panagiotou2014}.

\section{Acknowledgments} Eleni Panagiotou was supported by NSF (Grant No. DMS-1913180). Louis Kauffman was supported by the 
Laboratory of Topology and Dynamics, 
Novosibirsk State University 
(contract no. 14.Y26.31.0025 
with the Ministry of Education and Science 
of the Russian Federation.)

\bibliographystyle{plain}

\bibliography{main}

\end{document}